\documentclass[12pt,reqno]{amsart}

\usepackage{amssymb,amsthm}
\usepackage{amsfonts,cmtiup,comment,stmaryrd}

\usepackage{mathrsfs,cmtiup}

\makeatletter
\def\blfootnote{\xdef\@thefnmark{}\@footnotetext}
\makeatother

\newtheorem*{theorems}{Theorem A}
\newtheorem*{corollaryv}{Corollary B}
\newtheorem*{corollaryv2}{Corollary C}
\newtheorem*{proposition*}{Proposition}
\newtheorem{theorem}{Theorem}[section]
\newtheorem{lemma}[theorem]{Lemma}
\newtheorem{proposition}[theorem]{Proposition}
\newtheorem{corollary}[theorem]{Corollary}

\theoremstyle{definition}
\newtheorem{example}[theorem]{Example}

\newtheorem*{definition*}{Definition}
\newtheorem*{example*}{Example}
\numberwithin{equation}{section}

\begin{document}

\title{Local--global generation property of commutators\\ in finite $\pi$-soluble groups}

\dedicatory{\large To the memory of Marty Isaacs}

\author{Cristina Acciarri}
\address{C.~Acciarri: Dipartimento di Scienze Fisiche, Informatiche e Matematiche, Universit\`a degli Studi di Modena e Reggio Emilia, Via Campi 213/b, I-41125 Modena, Italy}
\email{cristina.acciarri@unimore.it}

\author{Robert M. Guralnick}
\address{Robert M. Guralnick: Department of Mathematics, University of
Southern California, Los Angeles, CA90089-2532, USA}
\email{guralnic@usc.edu}

\author{Evgeny Khukhro}
\address{E. I. Khukhro: Charlotte Scott Research Centre for Algebra, University of Lincoln, U.K.}
\email{khukhro@yahoo.co.uk}

\author{Pavel Shumyatsky}

\address{P. Shumyatsky: Department of Mathematics, University of Brasilia, DF~70910-900, Brazil}
\email{pavel@unb.br}

\thanks{The first author is a member of ``National Group for Algebraic and Geometric Structures, and Their Applications'' (GNSAGA–INdAM). The second author was partially supported by NSF grant DMS 1901595 and Simons Foundation Fellowship 609771.
The third author was partially supported by the International Center for Mathematics at SUSTech in Shenzhen. The fourth author was partially supported by  FAPDF and CNPq.}

\begin{abstract}
 For a group $A$ acting by automorphisms on a group $G$, let $I_G(A)$ denote the set of commutators $[g,a]=g^{-1}g^a$, where $g\in G$ and $a\in A$, so that $[G,A]$ is the subgroup generated by $I_G(A)$. We prove that if $A$ is a $\pi$-group of automorphisms of a $\pi$-soluble finite group $G$ such that any subset of $I_G(A)$ generates a subgroup that can be generated by $r$ elements, then the rank of $[G,A]$ is bounded in terms of $r$. Examples show that such a result does  not hold without the assumption of $\pi$-solubility.  Earlier we obtained this type of results
 for groups of coprime automorphisms and for Sylow $p$-subgroups of $p$-soluble groups.
\end{abstract}

\maketitle

\section*{Introduction}

A `local--global' generation property with respect to cardinality is well-known for commutators in any, not necessarily finite, group: if there are only finitely many, say, $m$, commutators in a group~$G$, then $[G,G]$ is finite of order bounded in terms of~$m$. The proof follows from Schur's theorem about a group with finite central quotient.
A similar application of Schur's theorem proves the same property for a group of
automorphisms $A$ of a group $G$: if the set of commutators $[g,a]=g^{-1}g^a$ is finite of cardinality $m$,  then  $|[G,A]|$ is finite of order bounded in terms of $m$;
see Proposition~\ref{p-vved} at the end of the Introduction.

In this paper we prove `local--global' generation properties of commutators with respect to rank, extending the results of \cite{agks,tams} to non-coprime automorphisms. Here, the rank of a finite group $G$ is the least positive integer $r$ such that every subgroup of $G$ can be generated by $r$ elements. (This parameter is also called the  Pr\"ufer rank.) For a group of automorphisms (or a subgroup) $A$ of a group $G$, let $I_G(A)=\{[g,a]\mid g\in G,\; a\in A\}$ be the set of commutators  $[g,a]=g^{-1}g^a$.   The subgroup generated by $I_G(A)$ is precisely the commutator subgroup $[G,A]$. Recall that, for a set of primes $\pi$, a finite group is said to be $\pi$-soluble if it has a normal series whose factors are either $\pi'$-groups or soluble $\pi$-groups.

\begin{theorems}\label{t-vsol} Suppose that $A$ is a $\pi$-group of automorphisms of a finite $\pi$-soluble group $G$ such that any subset of $I_G(A)$ generates a subgroup that can be generated by $r$ elements. Then the rank of $[G,A]$ is bounded in terms of $r$.
\end{theorems}

The hypothesis that any subset of $I_G(A)$ generates an $r$-generator subgroup can be regarded as a rank analogue of the restriction on the cardinality of $I_G(A)$. Theorem~A means that the `local' condition of $r$-generation by subsets of this generating set $I_G(A)$ implies the `global' $f(r)$-generation of \emph{all subgroups} of $[G,A]$. Earlier we proved in \cite[Theorem~1.4]{agks} the same kind of result without the $\pi$-solubility assumption for $G$ but for a group of coprime automorphisms~$A$ (that is, for  $(|G|,|A|)=1$). Examples in \cite[\S\,5]{agks} show that in Theorem~A the $\pi$-solubility condition on $G$ cannot be dropped for a non-coprime $\pi$-group of automorphisms~$A$.
The proof of Theorem~A depends on the classification of finite simple groups.

When $\pi$ is the set of all primes, then $\pi$-solubility means solubility, while any group is then a $\pi$-group. Hence we have the following corollary.

\begin{corollaryv}\label{c-vved}
Suppose that $A$ is a group of automorphisms of a finite soluble group $G$ such that any subset of $I_G(A)$ generates a subgroup that can be generated by $r$ elements. Then the rank of $[G,A]$ is bounded in terms of $r$.
\end{corollaryv}

We also have a corollary about a  rank generation property of commutators with elements of $\pi$-subgroups in finite $\pi$-soluble groups.

\begin{corollaryv2}\label{c-vved2}
 Suppose that $H$ is a $\pi$-subgroup of a finite $\pi$-soluble group $G$ such that any subset of $I_G(H)$ generates a subgroup that can be generated by $r$ elements. Then the rank of $[G,H]$ is bounded in terms of $r$.
\end{corollaryv2}

Note  that Corollary~C is not equivalent to Theorem~A, since the group of automorphisms in this theorem does not have to be $\pi$-soluble. Earlier we proved in \cite[Theorem~1.1]{agks} this result in the case where $H$ is a Sylow $p$-subgroup of a finite $p$-soluble group $G$. Examples in \cite[\S\,5]{agks} show that the ($\pi$-)solubility condition on $G$ cannot be dropped in Corollaries~B,~C.

It is worth mentioning that a consequence of \cite[Theorem~1.3]{agks}  gives a bound in terms of $r$ for the rank of $[G,G]$ if every subset of the set of commutators in a finite group $G$ generates a subgroup that can be generated by $r$ elements, and this result is valid for any finite group, without any ($\pi$-)solubility conditions (unlike Theorem~A and Corollaries~B,~C, where such conditions are unavoidable).

In many respects, the set $I_G(A)$ is dual to the set of fixed points $C_G(A)$. In particular, for an automorphism $\alpha$ we have $|I_G(\alpha)|=|G:C_G(\alpha)|$. In numerous important results nice properties of a finite group are derived from various smallness conditions on the fixed-point subgroups of automorphisms. Many of these results stem from the seminal papers of J.~G.~Thompson \cite{tho, tho64} and G.~Higman \cite{hig} on automorphisms with few fixed points. The results in the present paper  continue the line of research in the `dual' direction, when conditions
on the sets $I_G(A)$ are used to obtain results on the structure of the subgroup $[G,A]$; see \cite{agks,tams,AGS23,as}.

Preliminary material is collected in \S\,\ref{s-prel}. In \S\,2 we prove Theorem~A in the case where the group $G$ is nilpotent, by largely repeating with minor changes the arguments from \cite{agks} and \cite{tams}. Both in this and subsequent sections we use the theory of powerful $p$-groups developed by A.~Lubotzky and A.~Mann \cite{LM}. The special case of nilpotent $G$ is an important ingredient of the proof of the main Theorem~A. Another ingredient is \cite[Theorem~1.4]{agks} on groups of coprime automorphisms. Apart from this result, we are also using certain consequences of auxiliary lemmas in \cite{agks} and \cite{tams} about coprime automorphisms, which are proved in \S\,3.  In \S\,4 Theorem~A   is proved for $\pi=\{p\}$, that is, for a $p$-group of automorphisms of a $p$-soluble group, using inter alia the celebrated Hall--Higman theorem \cite{ha-hi} on the minimal polynomial of a $p$-element in a  $p$-soluble linear group in characteristic~$p$. At the beginning of~\S\,5, Theorem~A is firstly proved in the special case of a cyclic group of automorphisms, and then it is proved in full generality using the cyclic case and the case of a $p$-group of automorphisms.

We finish this introduction with a proposition  about a `local--global' generation property of commutators with respect to cardinality, which holds  for any group of automorphisms of any, not necessarily finite, group. We also produce an example showing that, even  in the case of a cyclic group of automorphisms $A=\langle\varphi\rangle$, it is not sufficient to impose a restriction on the cardinality of $I_G(\varphi)$, rather than on  the cardinality of $I_G(\langle \varphi\rangle)$. Henceforth we write ``$(a,b\dots)$-bounded'' to abbreviate ``bounded above by some function depending only on the parameters $a,b\dots$".

\begin{proposition}\label{p-vved}
Suppose that $A$ is a group of automorphisms (or a subgroup) of a group~$G$ such that  $I_G(A)$ is finite of cardinality $m$. Then $[G,A]$ is finite of order bounded in terms of $m$.
\end{proposition}

\begin{proof}
  Let $I_G(A)=\{[g_1,a_1],\dots ,[g_m,a_m]\mid g_i\in G,\;a_i\in A\}$. Let $H=\langle\langle g_1,\dots ,g_m\rangle^A \rangle$, which is the minimal $A$-invariant subgroup containing all $g_i$. Then $[G,A]=[H,A]$. For any $g\in G$, we have $|A:C_A(g)|\leqslant m$. Hence the intersection $\bigcap_{i=1}^m C_A(g_i)$ has finite $m$-bounded index in $A$. Let $A_0$ be the maximal normal subgroup of $A$ contained in $\bigcap_{i=1}^m C_A(g_i)$; then  $A_0$ has $m$-bounded index in $A$. Clearly, $A_0$ centralizes~$H$, so that the quotient $A/A_0$ naturally acts on $H$. Hence,  $[G,A]=[H,A]=[H,A/A_0]$. Thus, replacing $G$ with $H$, we can assume that $A$ is finite of $m$-bounded order.

Let $\varphi$ be an arbitrary element of $A$, and let $|\varphi|=n$. We know that $|G:C_G(\varphi)|\leqslant m$; let $N$ be a normal subgroup of $G$ of $m$-bounded index  contained in $C_G(\varphi)$. Then $[G,\varphi ]$ centralizes~$N$ (see Lemma~\ref{l-ker} below). Hence $[G,\varphi]$ has a central subgroup of finite $m$-bounded index. Then the derived subgroup $[G,\varphi]'$ is finite of $m$-bounded order by Schur's theorem \cite[Theorem~4.12]{rob}. An easy calculation shows that
 $[g,\varphi][g,\varphi]^{\varphi}\cdots [g,\varphi]^{\varphi^{n-1}}=1$. By the `linearity' of $\varphi$ on the abelian quotient $[G,\varphi]/[G,\varphi]'$ we obtain that $xx^{\varphi}\cdots x^{\varphi^{n-1}}=1$ for all $x\in [G,\varphi]/[G,\varphi]'$. Then the centralizer of $\varphi$ in this quotient has exponent $n$. Since $|I_G(\varphi)|\leqslant m$, this quotient is generated by $m$ elements. Hence this centralizer is finite of $(m,n)$-bounded order, while its index is at most $m$ by hypothesis. Thus,  $[G,\varphi]/[G,\varphi]'$ is finite of $(m,n)$-bounded order, and hence $[G,\varphi]$ is also finite of $(m,n)$-bounded order. The commutator subgroup $[G,A]$ is equal to the product of the normal subgroups $[G,\varphi]$ with $\varphi$ running over some generating set of $A$. Therefore  $[G,A]$ is finite of order bounded in terms of $m$ and $|A|$, and therefore in terms of $m$, since $|A|$ is $m$-bounded by assumption.
\end{proof}

The following example shows that in Proposition~\ref{p-vved},  in the case of a cyclic group of automorphisms $A=\langle\varphi\rangle$, it is not sufficient to impose a restriction on the cardinality of $I_G(\varphi)$, rather than on  the cardinality of $I_G(\langle \varphi\rangle)$.

\begin{example}\label{ex-v}
   Let $a$ be a splitting automorphism of prime order $p$ of a homocyclic group $B=\langle b_1\rangle\times \dots \times \langle b_{p-1}\rangle$ of exponent $p^n$ defined by $b_1^a=b_2$, $b_2^a=b_3$, \dots $b_{p-1}^a=(b_1\cdots b_{p-1})^{-1}$.
   Let $G=B\langle a\rangle$ and let $\varphi$ be an automorphism of $G$ such that $a^\varphi=ab_1$ and $C_G(\varphi)=B$. Then $|I_G(\varphi)|=|G:C_G(\varphi)|=|G:B|=p$, but $[a,\varphi]=b_1$, $[a^2,\varphi]=b_2b_1$, \dots , $[a^{p-1},\varphi]=b_{p-1}\cdots b_2b_1$. We see that $[G,\varphi ]=B$ has order $p^{n(p-1)}$, so it is not $|I_G(\varphi)|$-bounded when $n\to \infty$.
\end{example}

\section*{Acknowledgement}

The authors are grateful to the anonymous referees for careful reading and several helpful suggestions.

\section{Preliminaries}\label{s-prel}

Henceforth all groups considered in this paper are finite. Let $G$ be a group. We use the standard notation $\langle S\rangle$ for the (sub)group generated by a subset $S$, and $a^H=\{a^h=h^{-1}ah\mid h\in H\}$ for the set of conjugates of an element or subset $a$ by elements of a subset $H\subseteq G$; then $\langle S^G\rangle$ is the normal closure of $S$ in~$G$.

Recall that the (Pr\"ufer) rank of a finite group $G$ is the least positive integer $r$ such that every subgroup of $G$ can be generated by $r$ elements. The rank of an abelian, or more generally, nilpotent group is the maximum of the ranks of its Sylow subgroups. We also recall another well-known fact.

 \begin{lemma}\label{l-nil-rank}
 If $N$ is a nilpotent group of class $c$ generated by $k$ elements, then the rank of $N$ is bounded in terms of $k$ and $c$.
 \end{lemma}

\begin{proof}
  Since the $i$-th factor of the lower central series of $N$ is generated by at most $k^i$ elements \cite[Corollary~10.2.3.]{mhall}, the group $N$ has a central series of $(c,k)$-bounded length with cyclic factors. Hence  any subgroup of $N$ has such a series and therefore can be generated by $(c,k)$-bounded number of elements.
\end{proof}

We shall be using the following well-known result of A.~Lubotzky and A.~Mann \cite[Propositions~2.6 and 4.2.6]{LM}.

\begin{lemma}\label{l-24}
Let $p$ be a prime, and $G$ a group of exponent $p^k$ and of rank $m$. Then there is a number $s(k,m)$ depending only on $k$ and $m$ such that $|G|\leqslant p^{s(k,m)}$.
\end{lemma}

While the next lemma follows from the positive solution of the Restricted Burnside Problem, it can also be derived from Lemma~\ref{l-24} applied to Sylow subgroups.

\begin{lemma}\label{l-burn}
If $G$ is a group of rank $r$ and exponent $n$, then the order of $G$ is $(r,n)$-bounded.
\end{lemma}

The Fitting subgroup $F(G)$ is the largest nilpotent normal subgroup of a group $G$. The Fitting series of $G$ is defined as $F_1(G) = F (G)$, and by induction $F_{i+1}(G)$ is the inverse image of $F (G/F_i(G))$. If $G$ is soluble, then the least number $h$ such that $F_h(G) = G$ is called the Fitting height of $G$.
Another consequence of having a bound for the rank is also well known;  see for example, \cite[Lemma~2.4]{glasgow}.

\begin{lemma}\label{l-rfit}
The Fitting height of a soluble group is bounded in terms of its rank.
\end{lemma}

By a simple group we always mean a finite non-abelian simple group. By a semisimple group we mean a direct product of simple groups.
Recall that the generalized Fitting subgroup $F^*(G)$ of a finite group $G$ is the   product of the Fitting subgroup $F(G)$ and the characteristic subgroup $E(G)$, which is a central product of all subnormal quasisimple subgroups of $G$, that is, $E(G)=\prod Q_i$ over all $Q_i$ such that $Q_i$ is subnormal in $G$,  $Z(Q_i)\leqslant [Q_i,Q_i]$, and $S_i=Q_i/Z(Q_i)$ is a non-abelian simple group.  Then $[F(G),E(G)]=1$ and $E(G)/Z(E(G))\cong F^*(G)/F(G)$ is a semisimple group that is the direct product of the $S_i$. Acting by conjugation, the group $G$ permutes the factors $Q_i$ and $C_G(F^*(G))\leqslant F(G)$.

Let $S(G)$ denote the soluble radical of a group $G$, that is, the largest normal soluble subgroup. The following fact  (see, for example, \cite[Lemma~2.1]{khu-shu131}) is a well-known consequence of  Schreier's conjecture on solubility of outer automorphism groups of non-abelian finite simple groups confirmed by the classification.

\begin{lemma}\label{l-21}
Let  $L/S(G)=F^*(G/S(G))$ be the generalized Fitting subgroup of the quotient of a group $G$ by the soluble radical, and let $K$ be the kernel of the permutational action of $G$ on the set of subnormal simple factors of $L/S(G)$. Then $K/L$ is soluble.
\end{lemma}

The generalized Fitting series of a finite group $G$ is defined starting with the generalized Fitting subgroup $F^*_1(G)=F^*(G)$, and then by induction $F^*_{i+1}(G)$ is the inverse image of $F^*(G/F^*_{i}(G))$. The least number $h$ such that $F^*_h(G)=G$ is the generalized Fitting height $h^*(G) $ of $G$.
An analogue of Lemma~\ref{l-rfit} holds for the generalized Fitting height.

\begin{lemma}\label{l-212}
The generalized Fitting height of a finite group $G$ is bounded in terms of its rank.
\end{lemma}

\begin{proof}
Let $r$ be the rank of $G$. In the notation of Lemma~\ref{l-21}, the soluble radical $S(G)$  has $r$-bounded Fitting height by Lemma~\ref{l-rfit}, and so does the soluble quotient $K/L$. The number of subnormal simple factors of $L/S(G)$ is at most $r$, since each of them contains an involution. Therefore the quotient $G/K$ has $r$-bounded order. Hence, the result follows.
\end{proof}

If a group $A$ acts by automorphisms on a group $B$ we use the usual notation for the  commutators $[b,a]=b^{-1}b^a$ and the commutator subgroup $[B,A]=\langle [b,a]\mid b\in B,\;a\in A\rangle$, as well as for the centralizers $C_B(A)=\{b\in B\mid b^a=b \text{ for all }a\in A\}$ and $C_A(B)=\{a\in A\mid b^a=b \text{ for all }b\in B\}$. The automorphism induced by an automorphism $\varphi  $ of $B$ on the quotient by a normal $\varphi $-invariant subgroup is denoted by the same letter $\varphi $. Note that $[B,\varphi]=[B,\langle\varphi\rangle]$, as a special case of the following elementary lemma, which we may use without special references.

\begin{lemma}\label{l-gen}
Let $A$ be a group acting by automorphisms on a group $B$. If $A=\langle S\rangle$, then $[B,A]=\prod_{s\in S}[B,s]$.
\end{lemma}

\begin{proof}
Indeed, for each $s\in S$, the subgroup $[B,s]$ is normal in $B$, and their product is $A$-invariant, since it is invariant under every generator in $S$. The group $A$ acts trivially on the quotient by this product, since so does every generator in $S$. Hence $[B,A]\leqslant\prod_{s\in S}[B,s]$, while the reverse inclusion is obvious.
\end{proof}

We reproduce two more elementary lemmas.

\begin{lemma}\label{l-ker}
Let $A$ be a group acting by automorphisms on a group $G$.
 If $N$ is a normal subgroup of $G$ such that $N\leqslant C_G(A)$, then $[G,A]$ centralizes $N$.
\end{lemma}

\begin{proof}
The hypothesis means that $A$ is contained in the kernel $K$ of the action of the semidirect product $GA$ by conjugation on $N$. Since $K$ is a normal subgroup of $GA$, we also have $[G,A]\leqslant K$.
\end{proof}

Another lemma is well-known.

\begin{lemma}\label{l-20p} If $A$ is a $p$-group of automorphisms of a group $G$, then there is an $A$-invariant Sylow  $p$-subgroup of $G$.
\end{lemma}

An automorphism  $\alpha$ of a finite group $G$ is said to be coprime if the order of $\alpha$ is coprime to the order of $G$, that is, $(|G|,|\alpha |)=1$. The following lemma collects some well-known facts about coprime automorphisms of finite groups (see, for example, \cite{gore} and  \cite[Corollary~3.28]{isaacs}); we shall sometimes use these facts without special references.

\begin{lemma}\label{l-20} Let $A$ be a group of automorphisms of a group $G$.
\begin{itemize}

\item[\rm (a)] If $(|A|,|G|)=1$, then the group $G$ has an $A$-invariant Sylow $p$-subgroup for each prime $p\in\pi(G)$.

\item[\rm (b)] If $N$ is an $A$-invariant normal subgroup of $G$ such that $(|N|,|A|)=1$, then $C_{G/N}(A)=C_G(A)N/N$; in particular, if $(|A|,|G|)=1$, then $G=[G,A]C_G(A)$.
\end{itemize}
\end{lemma}

 If a group $A$ acts by automorphisms on an elementary abelian $p$-group  $E$, then $E$ can be regarded as an $\mathbb{F}_pA$-module, in which the vector addition is the group multiplication in $E$ and the scalar multiplication by $i\in \mathbb{F}_p=\{0,1,2,\dots,p-1\}$ is taking the  $i$-th power.

 A bound for the rank of a $p$-group implies a bound for the rank of a group of its  automorphisms. The following lemma follows, for example, from \cite[Lemmas~2.1, 2.2, 2.3]{glasgow}.

 \begin{lemma}\label{l-raut}
 The rank of a group of  automorphisms of a finite $p$-group of rank~$r$ is bounded in terms of~$r$.
\end{lemma}

We mention another folklore result about automorphisms of groups of given rank; see, for example, \cite[Lemma~2.13]{agks}.

\begin{lemma}\label{l-autsem}
If $A$ is a group of coprime automorphisms of a semisimple group  $H$ of rank~$r$, then $A$ has $r$-bounded order.
\end{lemma}

For a group $A$ acting by automorphisms on a group $G$, recall the definition of the subset $I_G(A)=\{[g,\alpha]\mid g\in G,\;\alpha\in A\}$, where the commutators $[g,\alpha] =g^{-1}g^{\alpha}$ are considered in the natural semidirect product $GA$. Clearly, $[G,A]=\langle I_G(A)\rangle$. We shall use the following elementary properties without special references.

 \begin{lemma}\label{l-IGA}
 Let $A$ be a group acting by automorphisms on a group $G$.
\begin{itemize}

\item[\rm (a)] If $N$ is an $A$-invariant normal subgroup of $G$, then $I_{G/N}(A)=\{gN \mid g\in I_G(A)\}$.

\item[\rm (b)]  For any $\alpha\in A$ we have $|I_G(\alpha)|=|G:C_G(\alpha)|$; in particular, $|I_G(A)|\geqslant |G:C_G(\alpha)|$.

    \item[\rm (c)] Suppose that any subset of $I_G(A)$ generates a subgroup that can be generated by $r$ elements. Then this  hypothesis is inherited by $I_S(B)$ for any $B$-invariant section $S$ of $G$ and for any subgroup $B\leqslant A$.
    \end{itemize}
 \end{lemma}

\section{The case of nilpotent groups}\label{s-nilp}

In this section we prove Theorem~A for groups of automorphisms of nilpotent groups, and this result will be used throughout subsequent sections.

\begin{theorem}\label{t-nilp} Suppose that $G$ is a nilpotent finite group admitting a group of automorphisms $A$ such that any subgroup generated by a subset of $I_G(A)$ can be generated by $r$ elements. Then $[G,A]$ has $r$-bounded rank.
\end{theorem}

The proof largely follows the arguments in~\cite{tams} and \cite{agks}. Recall the usual notation $Z_i(H)$ and $\gamma_i(H)$ for the $i$-th term of the upper and lower central series of a group $H$, respectively.

\begin{lemma}\label{l-22}
Let $G$ be a group admitting a group of automorphisms $A$. Let $p$ be a prime and suppose that $M$ is an $A$-invariant $p$-subgroup of $[G,A]$ that is normal in $G$ such that $|[M,A]|=p^m$ for some non-negative integer $m$. Then $M\leqslant Z_{2m+1}(O_p([G,A]))$.
\end{lemma}

\begin{proof} We use induction on $m$. If $m=0$, then $M\leqslant C_{G}(A)$ and therefore $M\leqslant Z([G,A])$ by Lemma~\ref{l-ker}. Now let $m\geqslant 1$. Set $K=O_p([G,A])$ to lighten the notation; note that $K$ is normal in $G$.
 If $M\leqslant Z(K)$, there is nothing to prove. If $M\not\leqslant Z(K)$, then the image of $M$ in $K/Z(K)$ has a non-trivial intersection with the centre of this quotient. In other words, $M\cap Z_2(K) \not\leqslant Z(K)$. Then Lemma~\ref{l-ker} implies that 
 $[M\cap Z_2(K), A]\neq 1$. It follows that $|[M/(M\cap Z_2(K)), A]|<|[M,A]|=p^m$. Indeed,
 $$
 [M/(M\cap Z_2(K)), A]=[M,A](M\cap Z_2(K))/(M\cap Z_2(K))\cong [M,A]/([M,A]\cap (M\cap Z_2(K))),
 $$
 where $[M,A]\cap (M\cap Z_2(K))\geqslant [M\cap Z_2(K), A]\ne 1$. Thus, $|[M/(M\cap Z_2(K)), A]|\leqslant p^{m-1}$.
 By induction, $M/(M\cap Z_2(K)) \leqslant Z_{2m-1}(K/(M\cap Z_2(K)))$, whence $M\leqslant Z_{2m+1}(K)$, as required.
 \end{proof}

Throughout the rest of this subsection, unless stated otherwise, $G$ is a $p$-group admitting a group of automorphisms $A$ such that any subgroup generated by a subset of $I_G(A)$ can be generated by $r$ elements.

\begin{lemma}\label{l-23} Suppose that $G$ is of prime exponent $p$ or of exponent~4. There exists a number $l(r)$ depending on $r$ only such that the rank of $[G,A]$ is at most $l(r)$.
\end{lemma}

\begin{proof}
Let $C$ be Thompson's critical subgroup of $[G,A]$ (see \cite[Theorem 5.3.11]{gore}). Note that $C$ is normal in $G$, since it is a characteristic subgroup of the normal subgroup $[G,A]$.
Observe that $[Z(C),A]$ is an $r$-generated abelian subgroup of exponent $p$ (or 4) and so the order of $[Z(C),A]$ is at most $p^r$ (or $2^{2r})$. By Lemma \ref{l-22}, $Z(C)$ is contained in $Z_{2r+1}([G,A])$ (or in $Z_{4r+1}([G,A])$). Since $[[G,A],C]$ is contained in $Z(C)$, we conclude that $C$ is contained in $Z_{2r+2}([G,A])$ (or in $Z_{4r+2}([G,A])$). Recall that $\gamma_{2r+2}([G,A])$ commutes with $Z_{2r+2}([G,A])$ and so in particular $\gamma_{2r+2} ([G,A])$ (respectively, $\gamma_{4r+2} ([G,A])$) centralizes $C$. By Thompson's theorem, $C_{[G,A]}(C)=Z(C)$. Thus $\gamma_{2r+2}([G,A])$ (respectively, $\gamma_{4r+2}([G,A])$) is contained in $Z(C)$, that is, the quotient $[G,A]/Z(C)$ is nilpotent of class $2r+1$ (respectively, of class $4r+1$). Since $Z(C)\leqslant Z_{2r+1}([G,A])$ (or $Z(C)\leqslant Z_{4r+1}([G,A])$), it follows that $[G,A]$ has $r$-bounded nilpotency class. Since $[G,A]$ is $r$-generated by hypothesis, by Lemma~\ref{l-nil-rank} the rank of $[G,A]$ is $r$-bounded, as desired.
\end{proof}

We will require the concept of powerful $p$-groups introduced by A.~Lubotzky and A.~Mann in \cite{LM}. A finite $p$-group $H$ is \emph{powerful} if and only if $[H,H] \leqslant H^p$ for $p\neq 2$ (or $[H,H]\leqslant H^4$ for $p = 2$).
Apart from the original paper \cite{LM}, information about the properties of powerful $p$-groups can also be found in the books \cite{DDM} or \cite{khukhu2}.

\begin{lemma}\label{l-powerful}
There exists a number $\lambda=\lambda(r)$ depending only on $r$ such that $\gamma_{2\lambda+1}([G,A])$ is powerful.
\end{lemma}

\begin{proof}
Let $s'(m)=s(1,m)$ if $p\ne 2$, and $s'(m)=s(2,m)$ if $p=2$ for the function $s(k,m)$ as in Lemma~\ref{l-24}, and let $l(r)$ be as in Lemma~\ref{l-23}. Let $\lambda=s'(l(r))$ and consider $N=\gamma_{2\lambda+1}([G,A])$. Note that $N$ is normal in $G$. In order to show that $N'\leqslant N^p$ (or $N'\leqslant N^4$ when $p=2$), we assume that $N$ is of exponent $p$ (or 4) and prove that $N$ is abelian.

Since the subgroup $N$ is of exponent $p$ (or $4$), the rank of $[N,A ]$ is at most $l(r)$ by Lemma~\ref{l-23}. Then $|[N,A ]|\leqslant p^{s'(l(r))}=p^\lambda$ by Lemma~\ref{l-24}, whence $N\leqslant Z_{2\lambda+1}([G,A])$ by Lemma~\ref{l-22}. Since $[\gamma_i([G,A]),Z_i([G,A])]=1$ for any positive integer $i$, we conclude that $N$ is abelian, as required.
\end{proof}

\begin{lemma}\label{l-25}
For any $i\geqslant 1$, there exists a number $m_i=m(i,r)$ depending only on $i$ and $r$ such that $\gamma_i([G,A])$ is an $m_i$-generated group.
\end{lemma}

\begin{proof}
Let $N=\gamma_i([G,A])$, which is a normal subgroup of $G$. We can pass to the quotient $G/\Phi(N)$ and assume that $N$ is elementary abelian. It follows that $|[N, A]|\leqslant p^r$.  Then $N\leqslant Z_{2r+1}([G,A])$ by Lemma \ref{l-22}, and therefore $[G,A]$ has nilpotency class bounded only in terms of $i$ and $r$. Since $[G,A]$ is $r$-generated, the rank of $[G,A]$ is also $(i,r)$-bounded by Lemma~\ref{l-nil-rank}. In particular, $N$ is $m_i$-generated for some $(i,r)$-bounded number $m_i$.
\end{proof}

We are now ready to prove Theorem \ref{t-nilp}.

\begin{proof}[Proof of Theorem~\ref{t-nilp}] Recall that $G$ is a nilpotent group admitting a group of automorphisms $A$ such that any subgroup generated by a subset of $I_G(A)$ can be generated by $r$ elements. We need to show that $[G,A]$ has $r$-bounded rank. The rank of $[G,A]$ is equal to the rank of $[P,A]$, where $P$ is some Sylow $p$-subgroup of $G$. Therefore it is sufficient to prove the theorem in the case where $G$ is a $p$-group, which is what we assume from now on.

 Let $s'(m)=s(1,m)$ if $p\ne 2$, and $s'(m)=s(2,m)$ if $p=2$ for the function $s(k,m)$ as in Lemma~\ref{l-24}, and let $l(r)$ be as in Lemma \ref{l-23}. Take $N=\gamma_{2\lambda+1}([G,A])$, where $\lambda=\lambda(r)=s'(l(r))$; note that $N$ is normal in $G$. Let $d$ be the minimum number of generators of $N$. The number $d$ is $r$-bounded by Lemma~\ref{l-25}, and $N$ is a powerful $p$-group by Lemma \ref{l-powerful}. By the properties of powerful $p$-groups (see, for example, \cite[Theorem 2.9]{DDM}) the rank of $N$ is equal to $d$ and therefore is $r$-bounded. Since the nilpotency class of $[G,A]/N$ is $r$-bounded (recall that $\lambda$ depends only on $r$) and $[G,A]$ is $r$-generated, the rank of $[G,A]/N$ is also $r$-bounded by Lemma~\ref{l-nil-rank}. Since the rank of $[G,A]$ is at most the sum of the ranks of $[G,A]/N$ and $N$, the result follows.
\end{proof}

\section{Coprime automorphisms}

Throughout this section, let $A$ be a group of \emph{coprime} automorphisms of a finite group $G$ such that any subset of $I_G(A)$ generates a subgroup that can be generated  by $r$ elements. One of the main theorems of \cite{agks} tells us that then $[G,A]$ has $r$-bounded rank.

\begin{theorem}[{\cite[Theorem~1.4]{agks}}]\label{t-coprime}
Suppose that $G$ is a finite group admitting a group of coprime automorphisms $A$ such that, for a positive integer $r$, any subgroup generated by a subset of $I_G(A)$ can be generated by $r$ elements. Then $[G,A]$ has $r$-bounded rank.
\end{theorem}

Apart from this result,
we shall also need some consequences of certain preparatory lemmas in \cite{agks} and \cite{tams}, which are proved in this section.
Although the next lemmas will be used here only for cyclic groups $A$ of coprime automorphisms, it is natural to produce general lemmas for arbitrary  $A$ when this comes at no extra cost. The following lemma is similar to \cite[Lemma~2.6]{tams}.

\begin{lemma}\label{l-552} Suppose that $G=[G,A]$. Suppose that $N$ is an $A$-invariant normal subgroup such that the quotient $G/N$ is generated by $s$ elements from $I_{G/N}(A)$ and $[N,A]$ is generated by $t$ elements from $I_N(A)$. Then $G$ can be generated by $s+t$ elements of $I_G(A)$.
\end{lemma}

\begin{proof} Let $[N,A]$ be generated by elements $b_1,\dots,b_t\in I_N(A)$, and let $G/N$ be generated by elements $\bar a_1,\dots,\bar a_s\in I_{G/N}(A)$. In accordance with Lemma~\ref{l-IGA}(a) choose some pre-images $a_1,\dots, a_s\in I_G(A)$ of $\bar a_1,\dots,\bar a_s$. We claim that $G=\langle a_1,\dots,a_s,b_1,\dots,b_t\rangle$.
We set $H=\langle a_1,\dots,a_s\rangle$, so that $G=NH$. Since $[N,A]$ is normal in $N$, the normal closure $ [N,A]^G= [N,A]^H$ is contained in $\langle a_1,\dots,a_s,b_1,\dots,b_t\rangle$. Since the image of $N$ in the quotient $G/ [N,A]^H$ is centralized by $A$, it is central in $G/[N,A]^H=[G,A]/[N,A]^H$ by Lemma~\ref{l-ker}, and therefore the image of $H$ becomes normal. Thus, $[N,A]^H H$ is normal in $G$. Obviously, $A$ acts trivially on $G/ [N,A]^H H=NH/ [N,A]^H H$. Since $G=[G,A]$, we conclude that $G= [N,A]^H H$, and the result follows.
\end{proof}

\begin{lemma}\label{l-5-nilp} Suppose that $A=\langle\varphi\rangle$ is cyclic,  $G=[G,A]$,  and $G$ is nilpotent. Then $G$ is generated by $r$  elements of $I_G(\varphi)$.
\end{lemma}

\begin{proof} By the Burnside Basis Theorem for every Sylow $q$-subgroup $Q$ of $G$ the subgroup $[Q,\varphi ]$ is generated by $r$ elements from $I_Q(\varphi  )$.
Let $G=Q_1\times\dots \times Q_m$ and let $Q_i=\langle [x_{i1},\varphi ],\dots ,[x_{ir},\varphi]\rangle$. Note that $[x_{1j}\cdots x_{mj},\varphi]=[x_{1j},\varphi ]\cdots [x_{mj},\varphi]$ for every $j$,  and the cyclic subgroup generated by this product of elements from $m$ different Sylow subgroups of $G$ contains each of these elements. Then the $r$ elements $[x_{1j}\cdots x_{mj},\varphi]$, where $j=1,\dots ,r$, generate $G$.
\end{proof}

\begin{lemma}\label{l-5} Suppose that $A=\langle\varphi\rangle$ is cyclic,  $G=[G,A]$,  and $G$ is soluble. Then $G$ is generated by $r$-boundedly many elements of $I_G(\langle\varphi\rangle)$.
\end{lemma}

\begin{proof} Since the rank of $G$ is $r$-bounded by Theorem~\ref{t-coprime}, the Fitting height $h(G)$ is $r$-bounded by Lemma~\ref{l-rfit}; therefore we can proceed by induction on $h(G)$.
First let $h(G)=1$, whence $G$ is nilpotent. Then
$G$ can be generated by $r$ elements from $I_G(\langle\varphi\rangle)$ by Lemma~\ref{l-5-nilp}. When $h(G)>1$, the quotient $G/F(G)$ is generated by $r$-boundedly many  elements of $I_{G/F(G)}(\langle\varphi\rangle)$ by the induction hypothesis, and $[F(G),\varphi ]$ is generated by $r$ elements of $I_{F(G)}(\langle\varphi\rangle)$ as shown above. The result follows by Lemma~\ref{l-552}.
\end{proof}

We denote by $x^A$ the orbit of an element $x$ of any $A$-invariant section under the action of~$A$; we call such orbits $A$-orbits for brevity.

\begin{lemma}\label{l-55} Suppose that $G=[G,A]$. Suppose that $N$ is an $A$-invariant normal subgroup such that the quotient $G/N$ is generated by the $A$-orbits of $s$ elements in $I_{G/N}(A)$ and $[N,A]$ is generated by the $A$-orbits of $t$ elements in $I_N(A)$. Then $G$ can be generated by the $A$-orbits of $s+t$ elements of $I_G(A)$.
\end{lemma}

\begin{proof} Let $[N,A]$ be generated by the $A$-orbits of elements $b_1,\dots,b_t\in I_N(A)$, and let $G/N$ be generated by the $A$-orbits of elements $\bar a_1,\dots,\bar a_s\in I_{G/N}(A)$. In accordance with Lemma~\ref{l-IGA}(a) choose some pre-images $a_1,\dots, a_s\in I_G(A)$ of $\bar a_1,\dots,\bar a_s$. We claim that $G=\langle a_1^A,\dots,a_s^A,b_1^A,\dots,b_t^A\rangle$.
We set $H=\langle a_1^A,\dots,a_s^A\rangle$, so that $G=NH$. Since $[N,A]$ is normal in $N$, the normal closure $ [N,A]^G= [N,A]^H$ is contained in $\langle a_1^A,\dots,a_s^A,b_1^A,\dots,b_t^A\rangle$. Since the image of $N$ in the quotient $G/ [N,A]^H$ is centralized by $A$, it is central in $G/[N,A]^H$ by Lemma~\ref{l-ker}, and therefore the image of $H$ becomes normal. Thus, $[N,A]^H H$ is normal in $G$. Obviously, $A$ acts trivially on $G/ [N,A]^H H=NH/ [N,A]^H H$. Since $G=[G,A]$, we conclude that $G= [N,A]^H H$, and the result follows.
\end{proof}

Recall that, as a consequence of the classification,  if a simple group $G$ admits a coprime group of automorphisms $A$, then $A$ is cyclic
(see \cite{GLS3}).

\begin{lemma}\label{l-5554} Suppose that $G=[G,A]$ and $G$ is semisimple. Then $G$ can be generated by $r$-boundedly many $A$-orbits of elements of $I_G(A)$.
\end{lemma}

\begin{proof} Let $G=S_1\times\dots\times S_l$ where the factors $S_i$ are simple. The group of automorphisms $A$ permutes the simple factors and the proof of \cite[Lemma~3.9]{agks} shows that there are at most $r$ orbits under this action. Therefore we can assume without loss of generality that $A$ transitively permutes the factors $S_i$. If $G$ is simple, then $A$ is cyclic and \cite[Lemma~2.9]{agks} tells us that $G$ is generated by two nilpotent subgroups $P_1$ and $P_2$ such that $[P_1,A]=P_1$ and $[P_2,A]=P_2$. Each of the subgroups $P_i$ is generated by at most $r$ elements from $I_G(A)$ by Lemma~\ref{l-5-nilp}, whence $G$ is generated by at most $2r$ such elements. Therefore we can assume that $l\geqslant 2$.

We use the fact each non-abelian simple group can be generated by two elements. Let $a,b$ generate $S_1$. Choose $\alpha\in A$ such that $S_1^\alpha=S_2$. Consider the elements $x_1=a^{-1}a^\alpha$, $x_2=b^{-1}b^\alpha$, and $x_3=ab((ab)^{-1})^\alpha$, which belong to $I_G(A)$. Let $K=\langle x_1^A,x_2^A, x_3^A\rangle$, which is an $A$-invariant subgroup.
We observe that $1\neq x_1x_2x_3=[a,b]\in S_1\cap K$. The projection of $K$ onto $S_1$ is the whole group $S_1$, since the projections of $x_1$ and $x_2$ onto $S_1$ are $a^{-1}$ and $b^{-1}$, which generate $S_1$. Hence the conjugacy class $[a,b]^K$ is equal to $[a,b]^{S_1}$ and therefore generates $S_1$. Thus, $S_1$ is contained in $K$. Since $K$ is $A$-invariant and $A$ transitively permutes the factors $S_i$, we must have $K=G$ and the result follows.
\end{proof}

\begin{lemma}\label{l-p-7}
Suppose that $A=\langle\varphi\rangle$ is cyclic,  $G=[G,A]$, and the group $G$ is soluble-by-semisimple-by-soluble. Then the group $G$ can be generated by $r$-boundedly many $A$-orbits of elements of $I_G(\langle\varphi\rangle )$.
\end{lemma}

\begin{proof}
By hypothesis there are normal subgroups $S\leqslant T$ such that $S$ and $G/T$ are soluble, and $T/S$ is semisimple. By choosing $S$ to be the soluble radical, and then $T$ the full inverse image of the generalized Fitting subgroup of $G/S$, we can assume that $S$ and $T$ are $A$-invariant. The soluble quotient $G/T$ can be generated by $r$-boundedly many elements of $I_{G/T}(\langle\varphi\rangle)$ by Lemma~\ref{l-5}. Then in the quotient $G/S$ the semisimple subgroup $[T/S,\varphi]$ can be generated by $r$-boundedly many $A$-orbits of elements of $I_{T/S}(\langle\varphi\rangle )$ by Lemma~\ref{l-5554}. Hence the quotient $G/S$ can be generated by $r$-boundedly many $A$-orbits of elements of $I_{G/S}(\langle\varphi\rangle)$ by Lemma~\ref{l-55}. In turn, the soluble subgroup $[S,\varphi]$ can be generated by $r$-boundedly many  elements of $I_{S}(\langle\varphi\rangle )$ by Lemma~\ref{l-5}. Hence the group $G$ can be generated by $r$-boundedly many $A$-orbits of elements from  $I_{G}(\langle\varphi\rangle)$ by Lemma~\ref{l-55}.
\end{proof}

\begin{proposition}\label{p-gen2}
Suppose that $G$ is a finite group admitting a coprime automorphism $\varphi$ such that every subset of $I_G(\langle\varphi\rangle)$ generates a subgroup that can be generated by $r$ elements. Then $[G,\varphi]$ can be generated by $r$-boundedly many elements from $I_G(\langle\varphi\rangle)$.
\end{proposition}

\begin{proof}
We can assume that $G=[G,\varphi]$. By Theorem~\ref{t-coprime} the rank of $G$ is $r$-bounded. Hence the order of the automorphism induced by $\varphi$ on $F^*(G)/F(G)$ is $r$-bounded by Lemma~\ref{l-autsem}.  The order of the automorphism induced by $\varphi$ on $G/F^*(G)$ is $r$-bounded by \cite[Proposition~3.13]{agks}.
As a result, the order of the automorphism $\bar \varphi$  induced by $\varphi$ on $G/F(G)$ is also $r$-bounded. To lighten the notation, let $H=G/F(G)$.

Note that $H=[H,\varphi]$ since $G=[G,\varphi]$ by our assumption. Since the rank of $G$ is $r$-bounded,  the number of simple factors in $F^*(G)/F(G)$ is also $r$-bounded. This implies that $H$ has a subgroup $K$ of $r$-bounded index that is semisimple-by-soluble. Namely, $K$ can be chosen to be the kernel of the permutational action of $H$ on the set of simple factors in $F^*(G)/F(G)$. Then the quotient of $K$ by $F^*(G)/F(G)$ is soluble, because $F^*(G)/F(G)$ contains its centralizer in $H$ and the outer automorphism groups of simple groups are soluble, which fact (Schreier conjecture) follows from the classification (see \cite{GLS3}).  With this choice,
$K$ is a normal and $\varphi$-invariant semisimple-by-soluble subgroup of $r$-bounded index in $H$.

We claim that $H$ is generated by $r$-boundedly many elements from $I_{H}(\langle\varphi\rangle)$. Since $H=[H,\varphi ]$, the quotient $H/K$, being of $r$-bounded order, is of course generated by $r$-boundedly many images of elements from $I_{H}(\langle\varphi\rangle)$, which are elements of  $I_{H/K}(\langle\varphi\rangle)$.
By Lemma~\ref{l-p-7} the subgroup $K$ is generated by $r$-boundedly many orbits of elements from $I_{K}(\langle\varphi\rangle)$ under the action of $\langle\bar\varphi\rangle$. Every element of such an orbit  is also an element of $I_{K}(\langle\varphi\rangle)$, since $[u,\varphi^i]^{\varphi^s}=[u^{\varphi^s},\varphi^i]\in I_{K}(\langle\varphi\rangle)$ for any  $u$ in $K$. Since the order $|\langle\bar\varphi\rangle|$  is $r$-bounded, we obtain that the subgroup $K$ is generated by $r$-boundedly many elements from $I_{K}(\langle\varphi\rangle)$. Applying now  Lemma~\ref{l-552} to $H$ and its normal $\varphi$-invariant subgroup $K$ we obtain that $H$ is generated by $r$-boundedly many elements from $I_{H}(\langle\varphi\rangle)$.

Thus, $H=G/F(G)$ is generated by $r$-boundedly many elements from $I_{G/F(G)}(\langle\varphi\rangle)$.
The subgroup $[F(G),\varphi]$ is generated by $r$ elements from $I_{F(G)}(\langle\varphi\rangle)$ by Lemma~\ref{l-5-nilp}.  Applying Lemma~\ref{l-552} we obtain that $G$ is generated by $r$-boundedly many elements from $I_{G}(\langle\varphi\rangle)$.
\end{proof}

\section{$p$-Groups of automorphisms of $p$-soluble groups}

In this section we prove Theorem~A in the special case where $\pi=\{p\}$, that is, for a $p$-group of automorphisms (or $p$-subgroup) of a $p$-soluble group.

But first we prove a lemma about the most general situation, which will be applied both in this and next sections. This lemma shows that in the proofs about the rank of $[G,A]$ we can pass to a subgroup of $r$-bounded index.

\begin{lemma}\label{l-index2}
Let $G$ be a finite group admitting a group of automorphisms $A$ such that every subset of $I_G(A)$ generates a subgroup that can be generated by $r$ elements.
If $H$ is an $A$-invariant subgroup of $r$-bounded index $m=|G:H|$ such that the rank of $[H,A]$ is $r$-bounded, then the rank of $[G,A]$ is $r$-bounded.
\end{lemma}

\begin{proof}
We can assume that $H$ is normal in $G$ by passing to the normal core of $H$, which remains $A$-invariant. Since $[H,A]$ is normal in $H$, it has at most $m$ conjugates in $G$, so that the normal closure $M$ of $[H,A]$ in $G$ is a product of at most $m$ subgroups, each of which is normal in $H$ and has $r$-bounded rank. Hence $M$ has $r$-bounded rank and it is sufficient to prove that the rank of $[G/M,A]$ is $r$-bounded. Let $\bar G=G/M$ for brevity. Since $A$ acts trivially on $\bar H$, the subgroup $[\bar G,A ]$ centralizes $\bar H$ by Lemma~\ref{l-ker}. Hence  $[\bar G,A ]$ has a central subgroup $[\bar G,A ]\cap \bar H$ of $r$-bounded index. By Schur's theorem \cite[Theorem~4.12]{rob}, the derived subgroup $[\bar G,A ]'$ has $r$-bounded order. The abelian derived quotient of $[\bar G,A ]$ has rank at most $r$, since it is $r$-generated by hypothesis. Hence, the result follows.
\end{proof}

Recall  that, for a prime $p$, the largest normal $p$-subgroup of a group $G$ is denoted by $O_p(G)$, and the largest normal $p'$-subgroup by $O_{p'}(G)$. The full inverse image of $O_p(G/O_{p'}(G))$ in~$G$ is denoted by $O_{p',p}(G)$, then $O_{p',p,p'}(G)$ is the inverse image of $O_{p'}(G/O_{p',p}(G))$, and so on.
A group $G$ is said to be $p$-soluble, if this series terminates with $O_{p',p,p'\dots}(G)=G$; in this case, the number of symbols $p$ in the subscript is called the $p$-length of $G$. Recall that if $G$ is $p$-soluble, then $O_{p',p}(G)/O_{p'}(G)$ contains its centralizer in  $G/O_{p'}(G)$.

The following proposition makes use of the celebrated Theorem~B of P.~Hall and G.~Higman  \cite{ha-hi} about $p$-soluble linear groups in characteristic $p$.

\begin{proposition}\label{p-hh2}
Let $p$ be a prime and let $A$ be a $p$-group of automorphisms of a finite $p$-soluble group $G$  such that every subset of $I_G(A)$ generates a subgroup that can be generated by $r$ elements. Let $F^*$ be the generalized Fitting subgroup of $G/ O_{p',p}(G)$.
\begin{itemize}
 \item[\rm (a)] The exponent of the group of automorphisms induced by $A$ on $F^*$ is $r$-bounded; in particular, if $[F^*,A]\ne 1$, then $p$ is $r$-bounded.
 \item[\rm (b)] The order of $[F^*,A]$ is $r$-bounded.
\end{itemize}
\end{proposition}

\begin{proof}
It is a routine exercise to check that $F^*$ is a $p'$-group faithfully acting on the Frattini quotient $V$ of the $p$-group $O_{p',p}(G)/O_{p'}(G)$.

(a) We need to prove that for any $\varphi\in A$ the order $p^m$ of the automorphism of $F^*$ induced by $\varphi$ is $r$-bounded. We can obviously assume that $p^m\ne 1$. Consider the action of the semidirect product $[F^*, \varphi^{p^{m-1}}]\langle\varphi\rangle$ on $V$ regarded as an $\mathbb{F}_p[F^*, \varphi^{p^{m-1}}]\langle\varphi\rangle$-module. Let $U$ be an irreducible factor of a composition series of $V$ on which $[F^*, \varphi^{p^{m-1}}]$ acts non-trivially; such $U$ exists due to the coprimeness of the action of $[F^*, \varphi^{p^{m-1}}]$ on $V$. Since $\varphi^{p^m}\in Z\big([F^*, \varphi^{p^{m-1}}]\langle\varphi\rangle\big)$ and $U$ is irreducible in characteristic $p$, in fact $\varphi^{p^m}$ acts trivially on $U$. Therefore the $p$-soluble group of linear transformations induced by $[F^*, \varphi^{p^{m-1}}]\langle\varphi\rangle$ on $U$ has no non-trivial normal $p$-subgroups and thus satisfies the hypotheses of the Hall--Higman Theorem~B \cite[Theorem~B]{ha-hi}. By this theorem the minimum polynomial of $\varphi$ as a linear transformation of $U$ has degree at least $p^m-p^{m-1}$. This degree is known to be equal to the (maximum) dimension of the span of an orbit of a vector $v\in V$ under the action of $\langle \varphi\rangle$ (see, for example, \cite[Lemma~2.6(d)]{khu-moe}). This span is generated by $v$, $[v,\varphi ]$, $[[v,\varphi] ,\varphi ]$, \dots . All these elements, possibly excepting~$v$, are images of elements of $I_G(\langle\varphi\rangle )$, and therefore the dimension of the span is at most $r+1$. Thus, $p^m-p^{m-1}\leqslant r+1$. Hence $p^m$ is $r$-bounded, and, in particular, $p$ is also $r$-bounded.

(b) We now prove that $|[F^*,A ]|$ is $r$-bounded. If $[F^*,A ]=1$, there is nothing to prove. So we assume that $[F^*,A ]\ne 1$; then by part~(a) the exponent $p^m\ne 1$ of the group of automorphisms induced by $A$ on $F^*$ is $r$-bounded, and, in particular, the prime $p$ is also $r$-bounded.

Consider the action of $F^*A$ on the Frattini quotient $V$ of $O_{p',p}(G)/O_{p'}(G)$, which is faithful for $F^*$.  Since the prime $p$ is $r$-bounded,
\begin{equation}\label{e-willfollow}
\text{the result will  follow if we prove that the rank of }[V,[F^*,A]]\text{ is } r\text{-bounded,}
\end{equation}
since   then the order of $[V,[F^*,A]]$ will also be $r$-bounded, and  $[F^*,A]$ acts faithfully on  $[V,[F^*,A]]$ due to coprime action.

 To lighten the notation, let $T=[F^*,A]$, so that $T=[T,A]$ is a $p'$-group. By Proposition~\ref{p-gen2}, for any $\psi\in A$  the subgroup $[\psi, T]$ is generated by $r$-boundedly many  elements from $I_{T}(\langle\psi\rangle )$, say, $[\psi^{k_1}, g_1],\dots ,[\psi^{k_{f(r)}}, g_{f(r)}]$ for some $r$-bounded number $f(r)$, which we assume to be the same for all $\psi\in A$ for convenience.
Since $V$ can  be regarded as a vector space over $\mathbb{F}_p$, the rank of its subgroups is also the dimension as $\mathbb{F}_p$-subspaces, so we can use the notation ``$\dim$'' for this rank and additive notation for products in $V$.

The proof of the proposition will essentially follow from the next two lemmas.

 \begin{lemma}\label{l-genhh2}
 Suppose that $W$ is a $TA$-invariant section of $V$, and $\psi $ an element of $A$.
 Then
 $$
 \dim [W,[\psi, T]]\leqslant 2f(r)\dim [W,\psi ].
 $$
 \end{lemma}

 \begin{proof}
 For any $x\in W$, $g\in T$, and $\sigma\in \langle\psi\rangle$ we have
 $$
 [[\sigma, g],x]=[\sigma^{-1}\sigma^{g},x]=[\sigma ^{-1},x]^{\sigma^{g}}+ [\sigma^{g},x]=[\sigma^{-1},x]^{\sigma^{g}}+ [\sigma,x^{g^{-1}}]^{g}.
 $$
 Since $x$ is an arbitrary element of $W$, we obtain $[W,[\sigma, g]]\leqslant [W,\sigma]^{\sigma^{g}}+ [W,\sigma]^{g}$.
 The right-hand side is a sum of two subgroups of the abelian group $W$ each having rank at most  $\dim [W,\psi]$, so this sum has rank at most $2\dim [W,\psi]$.  Hence,
 \begin{equation}\label{e-sigm}
  \dim [W,[\sigma, g]]\leqslant  2\dim [W,\psi].
 \end{equation}
 Let  $[\psi, T]=\langle [\psi^{k_1}, g_1],\dots ,[\psi^{k_{f(r)}}, g_{f(r)}]\rangle$.
 Then $$
[W,[\psi, T]]=\sum _{i=1}^{f(r)} [W,[\psi^{k_i}, g_i]],
 $$
 whence $[W,[\psi, T]]$ has rank at most $2f(r)\dim [W,\psi]$ by \eqref{e-sigm}.
 \end{proof}

We now prove the key lemma.

\begin{lemma}\label{l-hh2l}
The rank of $[V,T]$ is $r$-bounded.
\end{lemma}

\begin{proof}
Using the fact that $A$ is a $p$-group, we choose in it a central series with cyclic factors
$$
1=A_0<A_1<\cdots<A_n=A,
$$
and let $\psi_i\in A_i$ be some pre-images of generators of the cyclic groups $A_i/A_{i-1}$.
We now apply Lemma~\ref{l-genhh2} with $\psi$ being consecutively $\psi_{1},\dots ,\psi_{n}$ and with $W$ being the sections $V/U_i$  for the $TA$-invariant subspaces $U_i=[V,[A_i,T]]$. We see, for example, that $U_0=0$ and $U_1=[V,[\psi_{1}, T]]$. Note that $\dim U_1/U_0=\dim U_1\leqslant 2f(r)\dim [V/U_0,\psi _1]$ by Lemma~\ref{l-genhh2}, and $U_1$ is $TA$-invariant since $\psi_1\in Z(A)$. In general, for any $k$, we see that
\begin{equation}\label{e-sum}
  U_{k}=\sum_{i=1}^{k}[V ,[\psi_{i},T]].
\end{equation}
Indeed, in obvious induction, if $U_s=\sum_{i=1}^{s}[V,[\psi_{i},T]]$ by the induction hypothesis,  then  $U_{s+1}=\sum_{i=1}^{s+1}[V,[\psi_{i},T]]$, since $U_{s+1}=[V,[A_{s+1},T]]=[V,[A_s,T]]+[V,[\psi_{s+1},T]$ because  $A_{s+1}=A_s\cdot \langle\psi_{s+1}\rangle$.

By Lemma~\ref{l-genhh2} we have
\begin{equation}\label{e-dimhh1}
\dim (U_{k}/U_{k-1})=\dim ([V/U_{k-1},[\psi_{k},T]])\leqslant 2f(r)\dim [V/U_{k-1},\psi _{k}]
\end{equation}
for every $k=1,\dots ,n$. By \eqref{e-sum},
$$
U_n=\sum_{i=1}^{n}[V,[\psi_{i},T]]=[V,[A,T]]=[V,T],
$$
where the middle equation holds because $[A,T]=\prod_{i=1}^{n}[\psi_{i},T]$,  since $A=\langle\psi_1,\dots,\psi_n\rangle$.
 As a result, in view of \eqref{e-dimhh1}, we have
\begin{equation}\label{e-dimhh2}
 \dim [V,T]=\sum_{i=1}^{n}\dim U_i/U_{i-1}\leqslant 2f(r) \sum_{i=1}^{n}\dim [V/U_{i-1},\psi _{i}].
\end{equation}

We now estimate the right-hand side of \eqref{e-dimhh2}. Namely, we claim that
\begin{equation}\label{e-dimhh3}
 \sum_{i=1}^{n}\dim [V/U_{i-1},\psi _{i}]\leqslant \dim [V,A]\leqslant r.
  \end{equation}
  Indeed, first note that $[V/U_{i-1},\psi _{i}]\leqslant U_i/U_{i-1}$. For each $i$ by
  Lemma~\ref{l-IGA}(a) we can choose a basis of $[V/U_{i-1},\psi _{i}]$ to consist of the images of elements $b_{i1},\dots ,b_{ik_i}$ from $I_V(\langle \psi_i\rangle)\subseteq I_V(A)$.
Then all these elements $b_{ij}$, $i=1,\dots , n$, $j=1,\dots ,k_i$, together are linearly independent. Indeed, if $\sum_{i,j}\alpha_{ij}b_{ij}=0$, then, firstly, $\alpha_{n1}=\cdots =\alpha_{nk_n}=0$ because $b_{n1},\dots,b_{nk_n}$ are linearly independent modulo $U_{n-1}$, while all the other $b_{ij}$ for $i\leqslant n-1$ are contained in $U_{n-1}$. Then $\alpha_{n-1,1}=\cdots =\alpha_{n-1,k_{n-1}}=0$ because $b_{n-1,1},\dots,b_{n-1,k_{n-1}}$ are linearly independent modulo $U_{n-2}$, while all the other $b_{ij}$ for $i\leqslant n-2$ are contained in $U_{n-2}$. Proceeding in obvious induction we obtain that $\alpha_{ij}=0$ for all $i,j$. Since these $b_{ij}$ are linearly independent elements from $I_V(A)$, their number, which is exactly $\sum_{i=1}^{n}\dim [V/U_{i-1},\psi _{i}]$ is at most $\dim [V,A]\leqslant r$.

Combining \eqref{e-dimhh2} and \eqref{e-dimhh3}, we obtain $[V,T]\leqslant 2rf(r)$, so that the rank of $[V,T]$ is $r$-bounded.
\end{proof}

We now finish the proof of part (b) of the proposition. In accordance with \eqref{e-willfollow}, since both $p$ and $\dim [V,T]=\dim [V,[F^*,A]]$ are $r$-bounded,   the order of $[V,[F^*,A]]$ is $r$-bounded. Then the order of  $[F^*,A]$ is also $r$-bounded as this group acts faithfully on  $[V,[F^*,A]]$ due to coprime action.
 \end{proof}

The next lemma will enable induction on the $p$-length.

\begin{lemma}\label{l-plength2}
Let $G$ be a finite $p$-soluble group admitting a $p$-group of automorphisms $A $ such that every subset of $I_G(A)$ generates a subgroup that can be generated by $r$ elements. Then the $p$-length of $[G,A]$ is $r$-bounded.
\end{lemma}

\begin{proof}
 Let $P$ be an $A$-invariant Sylow $p$-subgroup of $G$. The rank of $[P,A]$ is $r$-bounded by Theorem~\ref{t-nilp}. Since this is a normal subgroup of $P$, its normal closure $N=\langle [P,A]^G\rangle$ in $G$ has $r$-bounded $p$-length by \cite[Lemma~2.7]{agks}. The group of automorphisms $A$ acts trivially on $PN/N$ and therefore trivially on $O_{p',p}(G/N)/ O_{p'}(G/N)$. By Lemma~\ref{l-ker}, then $[G,A]$ also acts trivially on $O_{p',p}(G/N)/ O_{p'}(G/N)$. Since $O_{p',p}(G/N)/ O_{p'}(G/N)$ contains its centralizer in the $p$-soluble group $(G/N)/ O_{p'}(G/N)$, we obtain that the image of $[G,A]N/N$ in $(G/N)/ O_{p'}(G/N)$ is contained in $O_{p',p}(G/N)/ O_{p'}(G/N)$ and therefore $[G,A]N/N$ has $p$-length~1. Since the $p$-length of $N$ is $r$-bounded,  hence the result.
\end{proof}

We are now ready to prove the main result of this section.

\begin{theorem}\label{t-psol}
Suppose that $A$ is a $p$-group of automorphisms of a finite $p$-soluble group $G$ such that any subset of $I_G(A)$ generates a subgroup that can be generated by $r$ elements. Then $[G,A]$ has $r$-bounded rank.
\end{theorem}

\begin{proof}
 By Lemma~\ref{l-plength2} the $p$-length of $[G,A]$ is $r$-bounded. Hence we can proceed by induction on this $p$-length. The base of induction is $p$-length 0, which means that $[G,A]$ is a $p'$-group. We claim that then $[G,A]=[[G,A],A]$. Indeed, consider one of generators of $[G,A]$, an element $[g,\varphi]$ for $g\in G$ and $\varphi\in A$, say, of order $|\varphi|=p^k\ne 1$. We have
 $$
 [g,\varphi][g,\varphi]^\varphi [g,\varphi]^{\varphi^2}\cdots [g,\varphi]^{\varphi^{p^k-1}}=1.
 $$
 Since $\varphi$ acts trivially on $[G,A]/[[G,A],A]$, this equation means that $[g,\varphi]^{p^k}\in [[G,A],A]$. Since $[g,\varphi]$ is in fact a $p'$-element, we obtain $[g,\varphi]\in [[G,A],A]$. This holds for every generator, so  $[G,A]= [[G,A],A]$ if $[G,A]$ is a $p'$-group. Then the rank of $[G,A]$ is $r$-bounded by Theorem~\ref{t-coprime}.

 We now assume that the $p$-length of $[G,A]$ is at least 1. Since $[G,A ]$ is a normal subgroup of $G$, we have $O_{p',p}(G)\cap [G,A]=O_{p',p}([G,A])$. Therefore the $p$-length of the image of $[ G,A]$ in $G/O_{p',p}(G)$ is less than the $p$-length of $[G,A]$. By the induction hypothesis, the rank of this image is $r$-bounded. We claim that, moreover, the order of the image of $[G,A ]$ in $G/O_{p',p}(G)$ is $r$-bounded.

 \begin{lemma}\label{l-order22}
 The order of the image of $[G,A ]$ in $G/O_{p',p}(G)$ is $r$-bounded.
 \end{lemma}

\begin{proof}
Let $\bar G=G/ O_{p',p}(G)$. Let $F^*$ be the generalized Fitting subgroup of $ [\bar G,A]$,
which is contained in the generalized Fitting subgroup of $ \bar G$, since $[ \bar G,A]$ is normal in $ \bar G$ (see \cite[Corollary~13.11(c)]{hb3}). Recall that the generalized Fitting subgroup of $\bar G$ is a $p'$-group; hence  $F^*$ is a $p'$-group. It follows from Proposition~\ref{p-hh2}(b) that the order of $[F^*,A]$ is $r$-bounded. Hence $C_{F^*}(A)$ has $r$-bounded index in $F^*$, as $F^*=[F^*,A]C_{F^*}(A)$ due to coprime action. Let $N$ be a normal subgroup of $F^*$ contained in $C_{F^*}(A)$ and having $r$-bounded index in~$F^*$. Then $(F^*)^n\leqslant N\leqslant C_{F^*}(A)$ for some $r$-bounded number $n$. Since the rank of $F^*$ is $r$-bounded by the induction hypothesis (as $F^*\leqslant [ \bar G,A]$), the quotient $F^*/(F^*)^n$ has $r$-bounded order by Lemma~\ref{l-burn}. The subgroup $(F^*)^n$ is normal in $ \bar G$ as a characteristic subgroup of a characteristic subgroup $F^*$ of a normal subgroup $[ \bar G,A]$. Since $A$ acts trivially on $(F^*)^n$, it follows that $[ \bar G, A]$ also acts trivially on $(F^*)^n$, that is, $(F^*)^n\leqslant Z\big([ \bar G, A]\big)$. The image $F^*/(F^*)^n$ of the generalized Fitting subgroup in the quotient by a central subgroup $(F^*)^n$ is the generalized Fitting subgroup of the quotient. Since the generalized Fitting subgroup contains its centralizer and $|F^*/(F^*)^n|$ is $r$-bounded, we obtain that $F^*$ has $r$-bounded index in $[\bar G,A]$. Then  $(F^*)^n$ also has $r$-bounded index in $[\bar G,A]$. Thus it remains to prove that the order of $(F^*)^n$ is $r$-bounded.

 Since $(F^*)^n\leqslant Z\big([ \bar G, A]\big)$, we obtain that $Z\big([ \bar G, A]\big)$ has $r$-bounded index in $[\bar G,A]$. By Schur's theorem \cite[Theorem~4.12]{rob}, then the derived subgroup $[ \bar G,A]'$ of $[ \bar G,A]$ has $r$-bounded order.

 For any $g\in \bar G$ and $a\in A$ with $|a|=p^k\ne 1$, we have
 $$
 [g,a][g,a]^a[g,a]^{a^2}\cdots [g,a]^{a^{p^k-1}}=1.
 $$
 By the `linearity' of $a$ on the abelian quotient $\overline{[ \bar G,A]}=[ \bar G,A]/[ \bar G,A]'$ we obtain that
 $$
 hh^a h^{a^2}\cdots h^{a^{p^k-1}}=1 \qquad \text{for any}\quad h\in \overline{[ \bar G,a]}.
 $$
It follows that $a$ acts fixed-point-freely on the Hall $p'$-subgroup $U_a$ of $\overline{[ \bar G,a]}$: if $h\in C_{\overline{[ \bar G,A]}}(a)$, the above equation becomes
 $$
 h^{p^k}=1,
 $$
so $h=1$ if it is a $p'$-element. Therefore, $U_a=[U_a,a]$.
Since $\overline{[ \bar G,A]}=\prod_{a\in A}\overline{[ \bar G,a]}$ and this group is abelian, the Hall $p'$-subgroup $U$ of this product is the product $U=\prod_{a\in A}U_a$. Since $U=C_U(A)\times [U,A]$ by coprime action, we must have $C_{U}(A)=1$, as otherwise $U_a=[U_a,a]\leqslant [U,A]\ne U$ for all $a\in A$, contrary to $U=\prod_{a\in A}U_a$.

 Since $(F^*)^n\leqslant C_{[ \bar G,A]}(A)$ and $F^*$ is a $p'$-group, it follows that
 $(F^*)^n\leqslant [ \bar G,A]'$. Since $[ \bar G,A]'$ has $r$-bounded order, as a result, $(F^*)^n$ has $r$-bounded order, and therefore so does $[ \bar G,A]$. \end{proof}

 \begin{lemma}\label{l-order-a}
 The group of automorphisms of $G/O_{p',p}(G)$ induced by $A$ has $r$-bounded order.
 \end{lemma}

\begin{proof}
Let $\bar G=G/ O_{p',p}(G)$. Let $F^*$ be the generalized Fitting subgroup of $ \bar G$, which is a $p'$-group.
By Proposition~\ref{p-hh2}(b) the order of $[F^*,A]$ is $r$-bounded. Hence $A_1:=C_A([F^*,A])$ has $r$-bounded index in $A$. Since $F^*=[F^*,A]C_{F^*}(A)$ due to coprime action, it follows that $A_1$ acts trivially on $F^*$. Then $[\bar G,A_1]$ also acts trivially on $F^*$ by Lemma~\ref{l-ker}. Since $F^*$ contains its centralizer in $\bar G$, it follows that $[\bar G,A_1]\leqslant F^*$. Since $F^*$ has coprime order to $|A|$, then $\bar G=F^*C_{\bar G}(A_1)$ by Lemma~\ref{l-20}(b), so that $\bar G=F^*C_{\bar G}(A_1)\leqslant C_{\bar G}(A_1)$. Thus, $A$ acts on $\bar G$ as a group $A/A_1$ of $r$-bounded order. \end{proof}

 We return to the proof of the theorem.
 Recall that for $\bar G=G/O_{p',p}(G)$ the order $|[\bar G,A]|$ is $r$-bounded by Lemma~\ref{l-order22}. Since $A$ acts on $\bar G$ as a group of $r$-bounded order $\bar A$ by Lemma~\ref{l-order-a}, it follows that $C_{\bar G}(A)$ has $r$-bounded index in $\bar G$. Indeed, for every $a\in \bar A$ the index $|\bar G:C_{\bar G}(a)|$ is $r$-bounded, since it is equal to the number of commutators $[g,a]$, $g\in \bar G$. Then $\bigcap_{a\in \bar A}C_{\bar G}(a)=C_{\bar G}(A)$ is a subgroup of $r$-bounded index in $\bar G$.

 Let $\bar C$ be the normal core of $C_{\bar G}(A)$ in $\bar G$, which is an $A$-invariant normal subgroup of $r$-bounded index in $\bar G$ contained in $C_{\bar G}(A)$.
 Let $C$ be the full inverse image of $\bar C$ in $G$. Since $C$ is $A$-invariant and has $r$-bounded index in $G$, by Lemma~\ref{l-index2} it is sufficient to prove that $[C,A]$ has $r$-bounded rank.
 Since $C$ is a normal subgroup of $G$ containing $O_{p',p}(G)$, we have $O_{p',p}(C)=O_{p',p}(G)$. First we consider $C/O_{p'}(C)$.

 \begin{lemma}\label{l-mainc22}
 The rank of $[C/O_{p'}(C),A]$ is $r$-bounded.
 \end{lemma}

 \begin{proof} Let $\bar C=C/O_{p'}(C)$ for brevity.
Let $P=O_p(\bar C)$. Since $[\bar C,A]\leqslant P$, by the Burnside Basis Theorem the subgroup $[\bar C,A]$ is generated by $r$ elements of the form $[a_1, c_1],\dots ,[a_r ,c_r]$ for some $a_i\in A$, $c_i\in \bar C$. For any $x\in P$, $a\in A$, and $c\in \bar C$ we have
 $$
[[a,c],x]=[a ^{-1}a^c,x]=[a ^{-1},x]^{a^c}[a^c,x]=[a ^{-1},x]^{a^c}[a,x^{c^{-1}}]^c.
$$
The right-hand side belongs to $[P,a]^{a^c}[P,a]^c$. This is a product of two normal subgroups of $P$ each having $r$-bounded rank by Theorem~\ref{t-nilp}. Therefore the rank of this product is $r$-bounded. Since this product does not depend on $x\in P$, we obtain
$$[[a,c],P]\leqslant [P,a]^{a^c}[P,a]^c,$$ so that $[[a,c],P]$ has $r$-bounded rank. Since $[\bar C,A]=\langle [a_1, c_1],\dots ,[a_r ,c_r]\rangle$, the commutator subgroup $[[\bar C,A ],P]$ is equal to the product
$$
[[\bar C,A ],P]=\prod_{i=1}^{r}[[a_i,c_i],P]
$$
and therefore has $r$-bounded rank.

 The subgroup $[[\bar C,A ],P]$ is normal in $\bar C$ and $A$-invariant. In the quotient $\hat{C}=\bar C/[[\bar C,A ],P]$ the subgroup $[\hat C,A]$ is centralized by $\hat P$ and therefore $[\hat C,A]\leqslant Z(\hat P)$, since $[\bar C,A]\leqslant P$. Hence $[\hat C,A]$ is abelian and has rank at most $r$ as it is $r$-generated by hypothesis. Since $[[\bar C,A ],P]$ has $r$-bounded rank, as a result, $[\bar C,A]$ also has $r$-bounded rank.
 \end{proof}

 We now proceed to prove that $[C,A]$ has $r$-bounded rank, which will complete the proof of the theorem.

 Let $P$ now denote an $A$-invariant Sylow $p$-subgroup of $O_{p',p}(C)$. We can assume that $P\ne 1$, since otherwise $C=O_{p'}(C)C_C(A)$ by Lemma~\ref{l-20}(b) and then $[C,A]=[O_{p'}(C),A]=[[O_{p'}(C),A],A]$ has $r$-bounded rank by Theorem~\ref{t-coprime}.
 We have $C=O_{p'}(C)N_C(P)$ by the Frattini argument. Our nearest goal is to show that $[N_C(P),A]$ is generated by $[N_C(P)\cap O_{p'}(C),A]$ and $r$-boundedly many elements from $I_{N_C(P)}(A)$. To lighten the notation, let $N=N_C(P)$ and $S=O_{p'}(C)\cap N$, so that $S\times P$ is a normal $A$-invariant subgroup of $N$ such that $A$ acts trivially on $N/( S\times P)$, that is, $[N,A]\leqslant S\times P$, and $S$ is a $p'$-group, while $P$ is a $p$-group.

 \begin{lemma}
 \label{l-normal22}
 We have $[N,A ]\cap S=[S,A]$; in particular, $[S,A]$ is a normal $A$-invariant subgroup of $N$.
 \end{lemma}

\begin{proof}
 Let $a\in A$ with $|a|=p^k\ne 1$. For any $n\in N$, we have $[n,a]=sh$ for $s\in S$ and $h\in P$. Since
 $$
 [n,a][n,a]^a\cdots [n,a]^{a^{p^k-1}}=1,
 $$
 we also have
 $$
 ss^a\cdots s^{a^{p^k-1}}=1.
 $$
 The image $\bar s$ in the quotient $S/[S,a]$ is centralized by $a$, and therefore the last equation turns into $\bar s^{p^{k}}=1$. Since $s$ is a $p'$-element, this means that $s\in [S,a]$. Then any product of elements of the form $[n,a]=sh$ for $s\in S$ and $h\in P$ that belongs to $S$ must  belong to the product of subgroups of the form $[S,a]$ contained in $[S,A]$. Thus, $[N,A ]\cap S\leqslant [S,A]$, and the reverse inclusion is obvious.

 Since both $S$ and $[N,A]$ are normal $A$-invariant subgroups of $N$, so is $[N,A ]\cap S=[S,A]$.
\end{proof}

\begin{lemma}
 \label{l-n-gen22}
 The subgroup $[N,A]$ is generated by $[S,A]$ and $r$-boundedly many elements from  $I_{N}(A)$.
\end{lemma}

\begin{proof}
 By Lemma~\ref{l-normal22} we can consider the quotient $\bar N=N/[S,A ]$ admitting the action of~$A$. We claim that $[\bar N,A]\leqslant \bar P$. Indeed, $[\bar N,A]\leqslant \bar S\times \bar P$, and since $(|\bar S|,|\bar P|)=1$, in fact, by Lemma~\ref{l-normal22}
 $$
 [\bar N,A]= ([\bar N,A]\cap \bar S)\times ([\bar N,A]\cap \bar P)=[\bar S,A]\times ([\bar N,A]\cap \bar P)=[\bar N,A]\cap \bar P,
 $$
 so that $[\bar N,A]\leqslant \bar P$. By hypothesis and by the Burnside Basis Theorem the subgroup $[\bar N,A]$ is generated by $r$ elements of the form $[a_1, \bar n_1],\dots ,[a_r ,\bar n_r]$ for some $\bar n_i\in \bar N$, $a_i\in A$.
 Let $n_i\in N$ be some fixed pre-images of the elements $\bar n_i$. By multiplying any element of $[N,A]$ by a suitable word in  $[a_1, n_1],\dots ,[a_r , n_r]$ we can obtain an element of $[S,A]$.
Hence, the result follows.
\end{proof}

We are now ready to finish the proof of the theorem. By Lemma~\ref{l-n-gen22} the subgroup $[N_C(P),A]$ is generated by $[S,A]$ and $r$-boundedly many elements from $I_{N_C(P)}(A)$, say, $[a_1, u_1],\dots ,[a_{f(r)} ,u_{f(r)}]$ for some $u_i\in N_C(P)$ and $a_i\in A$. For any $x\in O_{p'}(C)$, $a\in A$, and $u\in N_C(P)$ we have
 $$
[[a,u],x]=[a ^{-1}a^u,x]=[a ^{-1},x]^{a^u}[a^u,x]=[a ^{-1},x]^{a^u}[a,x^{u^{-1}}]^u.
$$
The right-hand side belongs to $[O_{p'}(C),a]^{a^u}[O_{p'}(C),a]^u$. This is a product of two normal subgroups of $O_{p'}(C)$ each having $r$-bounded rank by Theorem~\ref{t-coprime}. Therefore the rank of this product is $r$-bounded. Since this product does not depend on $x\in O_{p'}(C)$, we obtain $$[[a,u],O_{p'}(C)]\leqslant [O_{p'}(C),a]^{a^u}[O_{p'}(C),a]^u,$$
 so that $[[a,u],O_{p'}(C)]$ has $r$-bounded rank. Since
 $$
 [N_C(P),A]=\langle [S,A], [a_1, u_1],\dots ,[a_{f(r} ,u_{f(r)}]\rangle,
  $$
 the commutator subgroup $[[N_C(P),A ],O_{p'}(C)]$ is equal to the product
$$
[[N_ C(P),A ],O_{p'}(C)]=[[S,A],O_{p'}(C)]\cdot \prod_{i=1}^{f(r)}[[a_i,u_i],O_{p'}(C)].
$$
Note that, since $S\leqslant O_{p'}(C)$, the factor $[[S,A],O_{p'}(C)]$ is contained in the subgroup $[A,O_{p'}(C)]$, which has $r$-bounded rank by Theorem~\ref{t-coprime}. Thus, every factor in the above product of $f(r)+1$ normal subgroups of $O_{p'}(C)$ has $r$-bounded rank, and  therefore $[[N_ C(P),A ],O_{p'}(C)]$ has $r$-bounded rank.

 The subgroup $[[N_C(P),A ],O_{p'}(C)]$ is normal in $C=O_{p'}(C)N_C(P)$ and $A$-invariant. Since it has $r$-bounded rank, we can pass to the quotient and assume that $[[N_C(P),A ],O_{p'}(C)]=1$. We claim that under this assumption,
 $[O_{p'}(C),A]$ is $N_C(P)$-invariant and therefore normal in $C=O_{p'}(C)N_C(P)$. Indeed, for any $g\in N_C(P)$
 we have
 $$
 [O_{p'}(C),A]^g= [O_{p'}(C)^g,A^g]\leqslant [O_{p'}(C),A[A,g]]= [O_{p'}(C),A],
 $$
 since $[A ,g]$ centralizes $O_{p'}(C)$.

 The rank of $[O_{p'}(C),A]$ is $r$-bounded by Theorem~\ref{t-coprime}, so we can assume that this $A$-invariant normal subgroup is trivial. Then $O_{p'}(C)$ is centralized by $A$ and therefore also by $[C,A]$, whence $[C,A]\cap O_{p'}(C)\leqslant Z([C,A])$. Since the rank of the central quotient  $[C,A]/([C,A]\cap O_{p'}(C))\cong [C,A]O_{p'}(C)/O_{p'}(C)$ is $r$-bounded by Lemma~\ref{l-mainc22}, the rank of the derived subgroup $[C,A]'$ is $r$-bounded by the Lubotzky--Mann theorem \cite[Theorem~4.2.3]{LM}. The abelian derived quotient $[C,A]/[C,A]'$ has rank at most $r$ as it is $r$-generated by hypothesis. Hence the rank of $[C,A]$ is $r$-bounded.

 The theorem is proved.
\end{proof}

We now derive an important corollary that will be used in the next section.

\begin{corollary}\label{c-fh}
Suppose that $A$ is a $\pi$-group of automorphisms of a finite $\pi$-soluble group $G$ such that any subset of $I_G(A)$ generates a subgroup that can be generated by $r$ elements.
Then the generalized Fitting height of $[G,A]$ is $r$-bounded.
\end{corollary}

\begin{proof}
For every  prime $p\in \pi$, choose a Sylow $p$-subgroup $A_p$ of $A$. The rank of $[G,A_p]$ is $r$-bounded by Theorem~\ref{t-psol}. Hence the generalized Fitting height of $[G,A_p]$ is $r$-bounded by Lemma~\ref{l-212}. Since
 $$
 [G,A]=\prod_{p\in \pi}[G,A_p],
 $$
 where all subgroups $[G,A_p]$ are normal, the generalized  Fitting height of  $[G,A]$
 is at most the maximum of their generalized  Fitting heights, and so it is $r$-bounded.
\end{proof}

\section{$\pi$-Automorphisms of $\pi$-soluble groups}

In this section we prove the main Theorem~A about  $\pi$-groups of automorphisms of $\pi$-soluble groups. The key step is the case of a cyclic group of automorphisms. First we make three observations about `big' primes.

Let $\rho $ be an automorphism of prime-power order $p^k$ of a finite $p$-soluble group $G$ such that every subset of $I_G(\langle\rho\rangle )$ generates a subgroup that can be generated by $r$ elements.
By Proposition~\ref{p-hh2}(a), if $\rho$ acts non-trivially on $F^*(G/O_{p',p}(G))$, then the prime $p$ is $r$-bounded. In other words, there is an $r$-bounded number $f(r)$ such that if $p>f(r)$, then $\rho$ centralizes $F^*(G/O_{p',p}(G))$. Then $[G/O_{p',p}(G),\rho]$ also centralizes $F^*(G/O_{p',p}(G))$ by Lemma~\ref{l-ker}, and then $[G/O_{p',p}(G),\rho]\leqslant F^*(G/O_{p',p}(G))$. Since $F^*(G/O_{p',p}(G))$ is a $p'$-group, it follows by Lemma~\ref{l-20}(b) that $G/O_{p',p}(G)=F^*(G/O_{p',p}(G))C_{G/O_{p',p}(G)}(\rho)\leqslant C_{G/O_{p',p}(G)}(\rho)$, that is,
\begin{equation}\label{e-big1}
  [G,\rho ]\leqslant O_{p',p}(G)\qquad \text{if }|\rho |=p^k\text{ for a prime }p>f(r).
\end{equation}

Let $\sigma $ be a coprime automorphism of a finite group $H$ such that every subset of $I_H(\langle\sigma\rangle )$ generates a subgroup that can be generated by $r$ elements. By \cite[Proposition~3.13]{agks}, there is an $r$-bounded number $g(r)$ such that $\sigma ^{g(r)}$ acts trivially on $H/F^*(H)$. It follows that if $\sigma$ has prime-power order $p^k$ with $p>g(r)$, then $\sigma$ acts trivially on $H/F^*(H)$, that is,
\begin{equation}\label{e-big2}
[H,\sigma ]\leqslant F^*(H)\qquad \text{if }|\sigma  |=p^k\text{ for a prime }p>g(r).
\end{equation}

Let $\tau $ be a coprime automorphism of a finite semisimple group $S$ that is a direct product of $m$ simple groups. Suppose that every subset of $I_S(\langle\tau\rangle )$ generates a subgroup that can be generated by $r$ elements. By Theorem~\ref{t-coprime}, the rank of $[S,\tau]$ is $r$-bounded, whence, in particular, the number of simple factors in $[S,\tau]$ is at most some $r$-bounded number $h(r)$. It follows that if $\tau$ has prime-power order $p^k$ with $p>h(r)$, then $\tau$ normalizes every simple factor in  $S= [S,\tau]\times C_S(\tau)$, that is,
\begin{equation}\label{e-big3}
\tau \text{ normalizes every simple factor in }S \text{ if }|\tau  |=p^k\text{ for a prime }p>h(r).
\end{equation}
\begin{definition*}
We say that a prime $p$ is \emph{big} if $p>\max\{f(r),g(r),h(r)\}$ for the functions $f(r),g(r),h(r)$ in \eqref{e-big1}, \eqref{e-big2},  \eqref{e-big3}, and \emph{small} otherwise.
\end{definition*}

The following theorem is of course the special case  of the main Theorem~A  for a cyclic group of automorphisms, but it is a key result in the proof of the latter.

\begin{theorem}\label{t-sol}
Let $G$ be a finite $\pi$-soluble group admitting a $\pi$-automorphism $\varphi$ such that every subset of $I_G(\langle\varphi\rangle)$ generates a subgroup that can be generated by $r$ elements. Then the rank of $[G,\varphi]$ is $r$-bounded.
\end{theorem}

\begin{proof}
 For each prime $p$, let $\langle\varphi_p\rangle$ be a Sylow $p$-subgroup of $\langle\varphi\rangle$. Let
 $$\beta =
 \prod_{\substack{p\text{ is a big}\\\text{prime}}}\varphi_p.
 $$
 Then
 $$
 [G,\varphi]=[G,\beta ]\cdot \prod_{\substack{p\text{ is a small}\\\text{prime}}}[G,\varphi_p].
 $$
 The rank of each subgroup $[G,\varphi_p]$ is $r$-bounded by Theorem~\ref{t-psol}. Since the number of small primes
 is $r$-bounded, the product
 $$
 \prod_{\substack{p\text{ is a small}\\\text{prime}}}[G,\varphi_p]
 $$
has $r$-bounded rank.  Therefore it remains to prove that the rank of $[G,\beta ]$ is $r$-bounded. Hence it is sufficient to consider the case where $\varphi=\beta $, that is, the order of $\varphi$ is only divisible by big primes, and this is what we assume throughout what follows.

 \begin{lemma}\label{l-psiq}
  Let $\langle\psi\rangle$ be a  Sylow $q$-subgroup of $\langle\varphi\rangle$ with $|\psi|=q^k\ne 1$.  The image of $[G,\psi ]$ in $G/F^*(G)$ is a $q$-group, possibly trivial.
 \end{lemma}

 \begin{proof}
 Since $q$ is a big prime by our assumption, $\psi $ acts trivially on $G/O_{q',q}(G)$ by \eqref{e-big1}, that is,  $[G,\psi ]\leqslant O_{q',q}(G)$.   Since $q$ is a big prime, $\psi $ acts trivially on $O_{q'}(G)/F^*(O_{q'}(G))$ by \eqref{e-big2}, whence   $[G,\psi ]$ centralizes $O_{q'}(G)/F^*(O_{q'}(G))$ by Lemma~\ref{l-ker}. Since $F^*(O_{q'}(G))\leqslant F^*(G)$ and $F^*([G,\psi])=[G,\psi]\cap F^*(G)$, as a result, $[G,\psi ]/F^*([G,\psi])$ is a product of a central $q'$-subgroup centralized by $\psi$ and a $q$-subgroup.

 We claim that $[G,\psi ]/F^*([G,\psi])$ is a $q$-group, that is, $[G,\psi ]$ is contained in $QF^*([G,\psi])$, where $Q$ is a Sylow $q$-subgroup of the   inverse image of $F(G/F^*([G,\psi]))$.  For any $g\in G$ we have
 $$
 [g,\psi][g,\psi]^{\psi}\cdots [g,\psi]^{\psi^{q^k-1}}=1.
 $$
 Hence the image of $[g,\psi ]$ in $[G,\psi ]/QF^*([G,\psi])$, which is centralized by $\psi$ as mentioned above, has order dividing $q^k$, so must be trivial as it is a $q'$-element.
 \end{proof}

 \begin{lemma}\label{l-rankq}
The quotient $[G,\varphi]/F^*([G,\varphi])$ is a nilpotent group of $r$-bounded rank.
 \end{lemma}

 \begin{proof}
 We know that $[G,\varphi]=\prod_{q}[G,\varphi _q]$, where $\langle\varphi_q\rangle$ is a Sylow $q$-subgroup of $\langle\varphi\rangle$. For every~$q$, the image of $[G,\varphi _q]$ in $G/F^*(G)$ is a normal $q$-subgroup by Lemma~\ref{l-psiq}. Hence $[G,\varphi]/F^*([G,\varphi])=[G,\varphi]/(F^*(G)\cap [G,\varphi])\cong [G,\varphi]F^*(G)/F^*(G)$ is a nilpotent group. This group has $r$-bounded rank, since the rank of every Sylow $q$-subgroup $[G,\varphi _q]F^*(G)/F^*(G)$ is $r$-bounded by Theorem~\ref{t-psol}.
 \end{proof}

 We return to the proof of the theorem. By Lemma~\ref{l-rankq}, it remains to show that the rank of $F^*([G,\varphi])$ is also $r$-bounded. The bulk of the proof is
 the case where $F^*([G,\varphi])$ is nilpotent, that is, $F^*([G,\varphi])=F([G,\varphi])$. First we perform a reduction to this case.

 Suppose that $F^*([G,\varphi])\ne F([G,\varphi])$.
 Since $[G,\varphi]$ is a normal subgroup, we have  $F^*([G,\varphi])=[G,\varphi]\cap F^*(G)$ and $F([G,\varphi])=[G,\varphi]\cap F(G)$. The quotient  $L:=F^*(G)/F(G)$  is a $\pi'$-group which is a direct product $S_1\times\dots\times S_m$ of simple groups $S_i$  permuted by $\varphi$. Since $|\varphi|$ is only divisible by big primes by our assumption, in fact, $\varphi$ normalizes every $S_i$ by \eqref{e-big3}. We have $[L,\varphi]\ne 1$, since otherwise $[L,[G,\varphi ]]=1$ by Lemma~\ref{l-ker}, whence $L=1$ contrary to our assumption. Since $\varphi$ is a coprime automorphism of $L$, by Theorem~\ref{t-coprime} the subgroup $[L,\varphi]$ has $r$-bounded rank, so that,  in particular, it is a  product of an $r$-bounded number of factors $S_{i_j}$ such that $S_{i_j}=[S_{i_j},\varphi]$.

We  assume without loss of generality that $[L,\varphi]=S_1\times\dots\times S_{k(r)}$ for an $r$-bounded number $k(r)$.  For every $i=1,\dots ,k(r)$ there is a prime divisor $q_i$ of $|\varphi|$ such that the Sylow $q_i$-subgroup $\langle\varphi_{q_i}\rangle$ of $\langle\varphi\rangle$ acts non-trivially on $S_i$, so that $S_i=[S_i,\varphi_{q_i}]$ (these primes $q_i$ need not be all different).
Then, in particular, the image of $[G,\varphi_{q_i}]$ in $G/F(G)$ contains $S_i$. Note that each $[G,\varphi_{q_i}]$ is a normal $\varphi$-invariant subgroup of $G$. As a result, the image of the product
$$
\prod_{i=1}^{k(r)}[G,\varphi_{q_i}]
$$
in $G/F(G)$ contains $[L,\varphi]$. This means that in the quotient
$$
\bar G=G\Big/\prod_{i=1}^{k(r)}[G,\varphi_{q_i}]
 $$
 the subgroup $F^*([\bar G,\varphi])$ is nilpotent. Each of the subgroups $[G,\varphi_{q_i}]$ has $r$-bounded rank by Theorem~\ref{t-psol}, and since their number is $r$-bounded, their product also has $r$-bounded rank.  Hence it remains to prove that $[\bar G,\varphi]$ has $r$-bounded rank.

 In other words, with a change of notation, it is sufficient to prove that  $[G,\varphi]$ has $r$-bounded rank in the case where  $F^*([ G,\varphi])= F([ G,\varphi])$ is nilpotent, and we assume this for the rest of this section.

  It is sufficient to prove that the rank of any Sylow $p$-subgroup $P$ of $F([G,\varphi])$ is $r$-bounded.  The proof of this fact resembles the end of the proof of \cite[Theorem~1.1]{tams} using the lemmas from \S\,\ref{s-nilp} that are similar to corresponding  lemmas in \cite{agks,tams}. But the first step, showing that $P$ is generated by $r$-boundedly many elements, has to be different from the proof in \cite{tams}. This is not only because for non-coprime automorphism we cannot assume that $G=[G,\varphi]$, but also because it is unclear, and probably not true, if the quotient $[G,\varphi]/F([G,\varphi])$ can be generated by $r$-boundedly many elements from $I_G(\langle\varphi\rangle)$. However, once we show that $P$ is generated by $r$-boundedly many elements, the rest of the proof is similar to the end of the proof of \cite[Theorem~1.1]{tams}.

\begin{proposition}\label{p-genp}
The Sylow $p$-subgroup $P$ of $F([G,\varphi])$ is generated by $r$-boundedly many elements.
\end{proposition}

\begin{proof}
 We need to show that the elementary abelian $p$-group $V=P/\Phi (P)$ has $r$-bounded rank. Since $V$ can also be regarded as a vector space over $\mathbb{F}_p$, the rank of its subgroups is also the dimension as $\mathbb{F}_p$-subspaces, so we can use the notation ``$\dim$'' for this rank and additive notation for products in $V$. Note that $V$ is centralized by $F([G,\varphi])$, so that the group $G/F([G,\varphi])$ naturally acts on $V$. Recall also that for  a Sylow $q$-subgroup $\langle\varphi_q\rangle$ of $\langle\varphi\rangle$, the image of $[G,\varphi_q]$ in $[G,\varphi]/F([G,\varphi])$ is a $q$-group by Lemma~\ref{l-psiq} and $[G,\varphi]/F([G,\varphi])$ is a direct product of the images of these subgroups $[G,\varphi_q]$ by Lemma~\ref{l-rankq}. Let the bar denote the images in $\bar G=G/F([G,\varphi])$.

 \begin{lemma}\label{l-genbig}
 Suppose that $U$ is a $\bar G\langle\varphi\rangle$-invariant section of $V$. Let $\langle\psi\rangle $ be a Sylow $q$-subgroup of $\langle\varphi\rangle$ for some prime $q$ (which could also be $q=p$). Then
 $$
 \dim [U,\psi ]^{\bar G}\leqslant (2r+1)\dim [U,\psi ],
 $$
 where $[U,\psi ]^{\bar G}$ is the normal closure of $[U,\psi ]$ in the natural semidirect product $U\rtimes \bar G$.
 \end{lemma}
 \begin{proof}
 Since  $\psi\in \langle\varphi\rangle$ and $[\bar G,\psi]$ is a $q$-group by Lemma~\ref{l-psiq}, by the hypothesis and the Burnside Basis Theorem $[\bar G,\psi]$ is generated by $r$ elements of the form $[\psi, g_1],\dots ,[\psi, g_r]$ for some $g_i\in \bar G$. For any $x\in U$ we have
 $$
 [[\psi, g_i],x]=[\psi ^{-1}\psi^{g_i},x]=[\psi ^{-1},x]^{\psi^{g_i}}+ [\psi^{g_i},x]=[\psi ^{-1},x]^{\psi^{g_i}}+ [\psi,x^{g_i^{-1}}]^{g_i}.
 $$
 The right-hand side belongs to the sum $[U,\psi ]^{\psi^{g_i}}+ [U,\psi ]^{g_i}$ of two subgroups of the abelian group $U$ each having rank  $\dim [U,\psi ]$, so this sum has rank at most $2\dim [U,\psi ]$. Hence $[U,[\psi, g_i]]$ has rank at most $2\dim [U,\psi ]$. Since $[\bar G,\psi]=\langle [\psi ,g_1],\dots ,[\psi, g_r]\rangle$, we have
 $$
[U,[\psi, \bar G]]=\sum_{i=1}^{r} [U,[\psi, g_i]],
 $$
 so $[U,[\psi, \bar G]]$ has rank at most $2r\dim [U,\psi ]$. Note that $[U,[\psi, \bar G]]$ is $\bar G\langle\varphi\rangle$-invariant. Modulo $[U,[\psi, \bar G]]$ the subgroup $[U,\psi]$ is $\bar G$-invariant: for any $g\in \bar G$ and $x\in U$ we have
 $$
 \begin{aligned}
 [\psi,x]^g=[ \psi [\psi ,g],x^g]&=[\psi,x^g ]^{[\psi ,g]}+ [[\psi ,g],x^g]\\
 &=
 [\psi,x^g ]+ [[\psi,x^g ], [\psi ,g]] + [[\psi ,g],x^g]\\
 &\in [U,\psi]+ [U,[\psi, \bar G]].
  \end{aligned}
 $$
 Hence $[U,\psi]^{\bar G}\leqslant [U,\psi] + [U,[\psi, \bar G]]$, and therefore  $\dim [U,\psi]^{\bar G}\leqslant (2r+1)\dim [U,\psi ]$.
 \end{proof}

We return to proving Proposition~\ref{p-genp} about $r$-bounded generation of~$P$.
We now apply Lemma~\ref{l-genbig} with $\psi$ being consecutively $\psi_{1},\dots ,\psi_{n}$, where $\varphi=\psi_{1}\cdots \psi_{n}$ for elements $\psi_{i}$ of order $q_i^{k_i}$, after arbitrarily ordering all the primes $q_i$ dividing $|\varphi|$. Namely, we construct by induction $\bar G\langle\varphi\rangle$-invariant subspaces $U_i$ of $V$ as follows. We define $U_0=0$ and let $U_1=[V,\psi_{1}]^{\bar G}$. Note that $\dim U_1/U_0=\dim U_1\leqslant (2r+1)\dim [V/U_0,\psi _1]$ by Lemma~\ref{l-genbig} and $U_1$ is $\bar G\langle\varphi\rangle$-invariant. When $U_{k-1}$ is defined (which is $\bar G\langle\varphi\rangle$-invariant), we define $U_{k}$ to be the full inverse image in $V$ of $[V/U_{k-1},\psi_{k}]^{\bar G}$. We do this consecutively for $k=1,\dots ,n$, that is, for all primes $q_1,\dots ,q_n$ dividing $|\varphi|$. By Lemma~\ref{l-genbig} we have
\begin{equation}\label{e-sum20}
\dim (U_{k}/U_{k-1})\leqslant (2r+1)\dim [V/U_{k-1},\psi _{k}]
\end{equation}
for every $k=1,\dots ,n$.

For any $k$, we see that \begin{equation}\label{e-sum2}
U_{k}=\sum_{i=1}^{k}[V, \psi_{i}]^{\bar G}.
\end{equation}
Indeed, in obvious induction, if $U_s=\sum_{i=1}^{s}[V, \psi_{i}]^{\bar G}$ by the induction hypothesis,  then  by construction $U_{s+1}=U_s+[V, \psi_{s+1}]^{\bar G}=\sum_{i=1}^{s}[V, \psi_{i}]^{\bar G}+[V, \psi_{s+1}]^{\bar G}$.

Since $\langle\varphi\rangle=\langle\psi_1,\dots,\psi_n\rangle$, we have
$$
[V,\varphi]^{\bar G} =\sum_{i=1}^{n}[V,\psi_{i}]^{\bar G}=U_n
$$
 by \eqref{e-sum2}.
 As a result, using  \eqref{e-sum20} we obtain
\begin{equation}\label{e-dim}
 \dim [V,\varphi]^{\bar G}=\sum_{i=1}^{n}\dim U_i/U_{i-1}\leqslant (2r+1) \sum_{i=1}^{n}\dim [V/U_{i-1},\psi _{i}].
\end{equation}

We now estimate the right-hand side of \eqref{e-dim}.

\begin{lemma}\label{l-dimv}
We have $\sum_{i=1}^{n}\dim [V/U_{i-1},\psi _{i}]\leqslant \dim [V,\varphi]\leqslant r$.
\end{lemma}

\begin{proof}
First note that $[V/U_{i-1},\psi _{i}]\leqslant U_i/U_{i-1}$ for any $i$ by construction. For each $i$ by Lemma~\ref{l-IGA}(a)  we can choose a basis of $[V/U_{i-1},\psi _{i}]$ to consist of the images of elements $b_{i1},\dots ,b_{ik_i}$ from $I_V(\langle \psi_i\rangle)\subseteq I_V(\langle \varphi\rangle)$.  Then all these elements $b_{ij}$, $i=1,\dots , n$, $j=1,\dots ,k_i$, together are linearly independent. Indeed, if $\sum_{i,j}\alpha_{ij}b_{ij}=0$, then, firstly, $\alpha_{n1}=\cdots =\alpha_{nk_n}=0$ because $b_{n1},\dots,b_{nk_n}$ are linearly independent modulo $U_{n-1}$, while all the other $b_{ij}$ for $i\leqslant n-1$ are contained in $U_{n-1}$. Then $\alpha_{n-1,1}=\cdots =\alpha_{n-1,k_{n-1}}=0$ because $b_{n-1,1},\dots,b_{n-1,k_{n-1}}$ are linearly independent modulo $U_{n-2}$, while all the other $b_{ij}$ for $i\leqslant n-2$ are contained in $U_{n-2}$. Proceeding in obvious induction we obtain that $\alpha_{ij}=0$ for all $i,j$. Since these $b_{ij}$ are linearly independent elements from $I_V(\langle \varphi\rangle)$, their number, which is exactly $\sum_{i=1}^{n}\dim [V/U_{i-1},\psi _{i}]$ is at most $\dim [V,\varphi]\leqslant r$.
\end{proof}

We now finish the proof of Proposition~\ref{p-genp} about the number of generators of $P$. Combining \eqref{e-dim} with Lemma~\ref{l-dimv} we obtain
\begin{equation}\label{e-dimmm}
\dim [V,\varphi]^{\bar G}\leqslant 2r^2+r.
\end{equation}
Consider the quotient $\hat G=G\big/(O_{p'}(F([G,\varphi ]))\Phi(P))$, the images in which we denote by hat. Clearly, $\hat G$ acts on $\hat P$ as $\bar G$, and  $\hat P\cong V$ as $\mathbb{F}_p \bar G$-modules. The normal $\varphi$-invariant subgroup $[\hat P,\varphi]^{\bar G}$ of $\hat G$ has $r$-bounded rank (at most $2r^2+r$) by \eqref{e-dimmm}. In the quotient $\hat G/[\hat P,\varphi]^{\bar G}$ the image of $\hat P$ is centralized by $\varphi$ and therefore  $\hat P\leqslant Z([\hat G,\varphi])$ by Lemma~\ref{l-ker}. Since $[\hat G,\varphi]/\hat P$ has $r$-bounded rank by Lemma~\ref{l-rankq}, the rank of the derived subgroup $[\hat G,\varphi]'$ is also $r$-bounded by the Lubotzky--Mann theorem \cite[Theorem~4.2.3]{LM}. The abelian derived quotient of $[\hat G,\varphi]$  has rank at most $r$ by hypothesis. Thus, the rank of $[\hat G,\varphi]$ is $r$-bounded, and in particular the rank of $\hat P\cong P/\Phi (P)$ is $r$-bounded, which means that $P$ is generated by $r$-boundedly many elements.
\end{proof}

The rest of the proof of Theorem~\ref{t-sol}  is similar to the end of the proof of \cite[Theorem~1.1]{tams}.

\begin{lemma}\label{l-g-i-gen}
For any $i=1,2,\dots $ there is an $(r,i)$-bounded number $m_i=m_i(r,i)$ such that $\gamma_i(P)$ has $m_i$ generators.
\end{lemma}

\begin{proof}
We need to show that the elementary abelian $p$-group $W=\gamma_i(P)/\Phi (\gamma_i(P))$ has $(r,i)$-bounded rank. We have $| [W,\varphi ]|\leqslant p^r$ by hypothesis.
 In the quotient $\bar G=G/\Phi (\gamma_i(P))$ the subgroup $W$ is normal in $\bar G$ and is a $\varphi$-invariant $p$-subgroup of $[\bar G,\varphi]$. By Lemma~\ref{l-22}, then $W\leqslant Z_{2r+1}(O_p([\bar G,\varphi]))$; in particular, $W\leqslant Z_{2r+1}(\bar P)$. Since $W=\gamma_i(\bar P)$, it follows that the nilpotency class of $\bar P$ is at most $i+2r+1$. Using also a bound in terms of $r$ for the number of generators of $\bar P$ given by Proposition~\ref{p-genp}, we obtain that the rank of $\bar P$ is $(r,i)$-bounded by Lemma~\ref{l-nil-rank}, and, in particular, $W$ has $(r,i)$-bounded number of generators and therefore so does $\gamma_i(P)$.
 \end{proof}

\begin{lemma}\label{l-powerful2}
There exists a number $\lambda=\lambda(r)$ depending only on $r$ such that $\gamma_{2\lambda+1}(P)$ is powerful.
\end{lemma}

\begin{proof}
Let $s'(m)=s(1,m)$ if $p\ne 2$, and $s'(m)=s(2,m)$ if $p=2$ for the function $s(k,m)$ as in Lemma~\ref{l-24}, and let $l(r)$ be as in Lemma~\ref{l-23}. Let $\lambda=s'(l(r))$ and consider $N=\gamma_{2\lambda+1}(P)$. Note that $N$ is normal in $G$. In order to show that $N'\leqslant N^p$ (or $N'\leqslant N^4$ when $p=2$), we assume that $N$ is of exponent $p$ (or 4) and prove that $N$ is abelian.

Since the subgroup $N$ is of exponent $p$ (or $4$), the rank of $[N,\varphi ]$ is at most $l(r)$ by Lemma~\ref{l-23}. Then $|[N,\varphi ]|\leqslant p^{s'(l(r))}=p^\lambda$ by Lemma~\ref{l-24}, whence $N\leqslant Z_{2\lambda+1}(O_p([G,\varphi]))$ by Lemma~\ref{l-22}; in particular, $N\leqslant Z_{2\lambda+1}(P)$. Since $[\gamma_i(P),Z_i(P)]=1$ for any positive integer~$i$, we conclude that $N$ is abelian, as required.
\end{proof}

We now finish the proof of the theorem. For the $r$-bounded number $\lambda$ furnished by Lemma~\ref{l-powerful2}, let $N=\gamma_{2\lambda+1}(P)$, which is a powerful $p$-group by that lemma. Since it is generated by $r$-boundedly many elements by Lemma~\ref{l-g-i-gen}, it has $r$-bounded rank by the well-known property of powerful $p$-groups (see \cite[Theorem~2.9]{DDM}). Since the nilpotency class of $P/N$ is $r$-bounded and $P$ is generated by $r$-boundedly many elements by Proposition~\ref{p-genp}, the rank of $P/N$ is also $r$-bounded by Lemma~\ref{l-nil-rank}. Since the rank of $P$ is at most the sum of the ranks of $N$ and $P/N$, we obtain a bound for the rank of $P$ in terms of $r$. The theorem is proved.
\end{proof}

We are now ready to prove the main Theorem~A.

\begin{proof}[Proof of Theorem~{\rm A}]
Recall that $A$ is  a $\pi$-group of automorphisms of a finite $\pi$-soluble group~$G$ such that every subset of $I_G(A)$ generates a subgroup that can be generated by $r$ elements. We need to show that the rank of $[G,A]$ is $r$-bounded. First we perform reduction to the case where  $[G,A]$ is soluble.

By Corollary~\ref{c-fh} the generalized Fitting height of $[G,A]$ is $r$-bounded, so that $[G,A]$ has a normal $A$-invariant series of $r$-bounded length the factors of which are either nilpotent or semisimple. Let $L_1,\dots,L_{k(r)}$ be the semisimple factors of this series, for some $r$-bounded number $k(r)$. Since each $L_i$ is a $\pi'$-group, the commutator subgroup $[L_i,A]$ has $r$-bounded rank by Theorem~\ref{t-coprime}, and $C_A([L_i,A])=C_A(L_i)$ because of coprime action. By Lemma~\ref{l-autsem}, the order of $A/C_A(L_i)$ is $r$-bounded. Let
$$
B=\bigcap_{i=1}^{k(r)}C_A(L_i).
$$
Then $[G,B]$ is soluble, since $B$ centralizes each $L_i$, and therefore $[G,B]$ also centralizes each $L_i$ by Lemma~\ref{l-ker}.

The order of $A/B$ is  $r$-bounded, as this group embeds into the direct product of the $k(r)$  quotients $A/C_A(L_i)$ of $r$-bounded order. In particular,  $A/B$ is generated by $r$-bounded number of elements, say, $A/B=\langle a_1,\dots, a_{f(r)}\rangle$ for some $r$-bounded number $f(r)$. The group $A/B$ naturally acts on the quotient $G/[G,B]$. We have
$$
[G,A]/[G,B]=\prod_{i=1}^{f(r)}\big[G/[G,B],a_i\big].
$$
The rank of each of the normal subgroups $\big[G/[G,B],a_i\big]$ is $r$-bounded by Theorem~\ref{t-sol}. Hence the rank of $[G,A]/[G,B]$ is $r$-bounded. It remains to prove that the rank of $[G,B]$ is also $r$-bounded.

Recall that $[G,B]$ is a soluble group, and its   Fitting height is $r$-bounded by Corollary~\ref{c-fh}.
Therefore we can proceed by induction on the Fitting height of $[G,B]$.
We can take for the base of induction the case where $[G,B]=1$.

In the induction step,
the rank of $[G,B]/F([G,B])$ is $r$-bounded by the induction hypothesis, and it remains to show that the rank of $F([G,B])$ is $r$-bounded. The rank of $F([G,B])$ is equal to the rank of some Sylow $p$-subgroup  of $F([G,B])$. This Sylow $p$-subgroup of $F([G,B])$ is isomorphic to a subgroup of $P=O_{p',p}([G,B])/O_{p'}([G,B])$, so it is sufficient to prove that the rank of $P$ is $r$-bounded. Note that $O_{p'}([G,B])$ is a normal $B$-invariant subgroup, so we can pass to the quotient $G/O_{p'}([G,B])$ and assume that $O_{p'}([G,B])=1$, so that $P=O_p([G,B])$ is a normal $p$-subgroup of $G$.
Note that $P$ contains its centralizer in $[G,B]$ due to our assumption $O_{p'}([G,B])=1$.
The subgroup $[P,B]$ has $r$-bounded rank by Theorem~\ref{t-nilp}. Hence the group $B/C_B([P,B])$ has $r$-bounded rank by Lemma~\ref{l-raut}. Note also that $C_B([P,B])/C_B(P)$ is a $p$-group, since any $p'$-element centralizing both $[P,B]$ and $P/[P,B]$ must centralize $P$ due to coprime action.

The subgroup $[G,C_B(P)]$ centralizes $P$ by Lemma~\ref{l-ker} and therefore $[G,C_B(P)]\leqslant P$, so that $[G,C_B(P)]\leqslant Z(P)$. In particular, $[G,C_B(P)]$ is abelian and therefore has rank at most~$r$ being $r$-generated by hypothesis. Note that  $[G,C_B(P)]$ is a normal $B$-invariant subgroup. Hence it is sufficient to prove that $[\bar G,B]$ has $r$-bounded rank for $\bar G=G/[G,C_B(P)]$. Since $[\bar G,C_B(P)]=1$, the action of $B$ on $\bar G$ factors through to the natural action of $B/C_B(P)$. Since $C_B([P,B])/C_B(P)$ is a $p$-group, the rank of $[\bar G, C_B([P,B])]$ is $r$-bounded by Theorem~\ref{t-psol}. Note that $[\bar G, C_B([P,B])]$  is a normal $B$-invariant subgroup. Hence it is sufficient to prove that $[\widetilde  G,B]$ has $r$-bounded rank for $\widetilde  G=\bar G/[\bar G,C_B([P,B])]$. Since $[\widetilde  G,C_B([P,B])]=1$, the action of $B$ on $\widetilde  G$ factors through to the natural action of $B/C_B([P,B])$. Since  $B/C_B([P,B])$ has $r$-bounded rank, it is in particular generated by $r$-boundedly many elements, say, $b_1,\dots,b_{f(r)}$. Each of the subgroups $[\widetilde G,b_i]$ has $r$-bounded rank by Theorem~\ref{t-sol}. Since
 $$
 [\widetilde G,B]=\prod_{i=1}^{f(r)}[\widetilde G,b_i],
 $$
 the rank of $[\widetilde G,B]$ is $r$-bounded. This completes the proof of the  induction step. The theorem is proved.
\end{proof}

\end{document}